\documentclass[11pt]{article}
\usepackage{subfigure}
\usepackage{authblk}
\usepackage{hyperref}
\usepackage{float, array}
\usepackage{mathtools}
\usepackage{amssymb,amsthm,amsmath,mathrsfs,fullpage,multirow}
\usepackage{tikz}
\usetikzlibrary{matrix}
\usepackage[pagewise]{lineno}

\usetikzlibrary{arrows.meta, positioning}

\makeatletter
\def\keywords{\xdef\@thefnmark{}\@footnotetext}
\makeatother

\DeclareMathOperator{\ind}{ind}

\renewcommand\S{\mathcal S}
\newcommand\M{\mathcal M}

\newtheorem{theorem}{Theorem}[section]
\newtheorem{lemma}[theorem]{Lemma}
\newtheorem{proposition}[theorem]{Proposition}

\theoremstyle{definition}
\newtheorem{definition}[theorem]{Definition}
  
\newtheorem{remark}[theorem]{Remark}
\newtheorem{example}[theorem]{Example}

\newcommand{\thistheoremname}{}
\newtheorem*{genericthm*}{\thistheoremname}
\newenvironment{namedthm*}[1]
  {\renewcommand{\thistheoremname}{#1}%
   \begin{genericthm*}}
  {\end{genericthm*}}

\usepackage{enumerate}

\title{Seaweed algebras with restricted part sizes}
\author[1]{Kassie Archer}
\author[1]{Aaron Geary}
\author[1]{Robert P. Laudone}
\affil[1]{{\small Department of Mathematics, United States Naval Academy, Annapolis, MD, 21402}}
\affil[ ]{{\small Email: \{karcher, geary, laudone\}@usna.edu }}
\date{}
\begin{document}

\maketitle

\begin{abstract}
Seaweed algebras are a class of Lie algebras that are naturally parameterized by a pair of compositions, which in turn are represented visually as planar graphs called meanders. These meanders provide a straightforward method for computing the index of the associated algebra. The goal of this paper is to enumerate those seaweed algebras with a fixed index and whose associated compositions have restricted part sizes. In particular, we enumerate those with composition part sizes coming from so-called acyclic sets. 
We also establish a bijection between sets of indecomposable seaweed algebras defined by compositions with certain restricted part sizes and sets of permutations with restricted displacements. In certain cases, the index of the algebra can be determined by a simple statistic on the permutation. 

\end{abstract}


\section{Introduction and Background}

Seaweed algebras, first introduced by Dergachev and Kirillov \cite{DK00}, form a family of subalgebras of $\mathfrak{gl}_n$, the general linear Lie algebra\footnote{There are various types of seaweed algebras, some of which are defined as subalgebras of the special linear Lie algebra $\mathfrak{sl}_n$. However, in this paper, we almost always consider those that are subalgebras of $\mathfrak{gl}_n$.}. These seaweed algebras are parameterized by a pair of compositions $\lambda = (\lambda_1, \ldots, \lambda_r)$ and $\mu=(\mu_1,\ldots, \mu_s)$ where $\sum_{i=1}^r \lambda_i=\sum_{i=1}^s \mu_i = n$. In particular, the seaweed algebra associated to the pair $(\lambda, \mu),$ which we denote by $\mathfrak{p}_n\frac{\lambda}{\mu}$, consists of exactly those matrices whose nonzero entries above (resp., below) the diagonal lie in blocks along the diagonal whose sizes are determined by the composition $\lambda$ (resp., $\mu$). An example of such a matrix can be found in Figure~\ref{fig: meander}. More about these algebras and how they are constructed can be found in \cite{Cet22}. 

The \emph{meander} associated to a seaweed algebra $\mathfrak{p}_n\frac{\lambda}{\mu}$ is a planar graph constructed in the following way:
\begin{itemize}
    \item Start with $n$ nodes arranged in a horizontal line.
    \item The upper and lower compositions of $n$, $\lambda$ and $\mu$, dictate how to draw arcs above and below the nodes. For each part $\lambda_i$, arcs connect pairs of nodes within that block in a nested fashion. Similarly, the composition $\mu$ determines the arcs drawn below the nodes.
\end{itemize}
For the purposes of this paper, we denote the meander itself by $\frac{\lambda}{\mu}.$
One reason that the meander associated to a seaweed algebra is so useful is that it provides a straightforward way to calculate the index of the seaweed algebra.


The \emph{index} of a Lie algebra is a fundamental invariant. For a given Lie algebra $\mathfrak{g}$, its index, denoted $\text{ind}(\mathfrak{g})$, is the minimum dimension of the stabilizers of elements in the dual space $\mathfrak{g}^*$ under the coadjoint action:
\[
\ind(\mathfrak{g}) = \min_{f \in \mathfrak{g}^*} \dim(\mathfrak{g}_f),
\]
where $\mathfrak{g}_f = \{x \in \mathfrak{g} \mid f([x, y]) = 0 \text{ for all } y \in \mathfrak{g}\}$. Calculating the index of a Lie algebra can generally be quite challenging. 
As shown in \cite{DK00}, the index of a seaweed algebra can be read from the structure of the associated meander. Let $c$ be the number of cycles and $p$ be the number of paths (including isolated nodes) in the meander associated with seaweed algebra $\mathfrak{p}_n\frac{\lambda}{\mu}$. Then we have that
\[
\ind(\mathfrak{p}_n\tfrac{\lambda}{\mu}) = 2c + p.
\]

For ease of notation, we set  $\ind(\frac{\lambda}{\mu}):= \ind(\mathfrak{p}_n\frac{\lambda}{\mu})$. Throughout this paper we will also refer to the index of a seaweed algebra as the index of its corresponding meander.

\begin{example} In Figure \ref{fig: meander}, we present a general matrix in $\mathfrak{p}_n\frac{\lambda}{\mu}$ together with the meander $\frac{\lambda}{\mu}$ where $\lambda=(5,3)$ and $\mu=(2,1,2,1,2)$. 
Blue is used for the upper arcs associated with $\lambda$ and red for the lower arcs associated with $\mu$. There is one cycle and two paths (including one degenerate path, an isolated node) in this meander, so we have $\text{ind}(\frac{\lambda}{\mu})=4$. 

\begin{figure}[h]
    \begin{center}
    \begin{tabular}{|>{\centering\arraybackslash}m{5.5cm}| >{\centering\arraybackslash}m{8cm}|}
    \hline 
    \begin{tikzpicture}
        \node at (0,0) {$\left[\begin{array}{rrrrrrrr}
            * & * & * &* & * & & & \\
             * & * & * &* & * & & & \\
              &  & * &* & * & & & \\
              &  &  &* & * & & & \\
              &  &  &* & * & & & \\
              & &  & &  & * & *&* \\
              & &  & &  &  & *&* \\
              & &  & &  & & *&* \\
        \end{array}\right]$};
        \draw [blue, thick, dotted] (-2.1,1.85)--(0.5,1.85)--(0.5,-.5)--(2.05,-.5)--(2.05,-1.9);
        \draw [red,thick, dotted] (-2.1,1.85)--(-2.1,1)--(-1.05,1) -- (-1.05,.5)--(-.5,.5)--(-.5,-.5)--(.5,-.5)--(.5,-1)--(1,-1)--(1,-1.9)--(2.05,-1.9);
    \end{tikzpicture}  &{
\begin{tikzpicture}[
    dot/.style={circle, fill, inner sep=1.5pt},
    arc/.style={line width=0.6mm}
]

    \draw[arc, blue] (1,0) to[bend left=60] (5,0);
    \draw[arc, blue] (2,0) to[bend left=45] (4,0);

    \draw[arc, blue] (6,0) to[bend left=60] (8,0);

    \draw[arc, red] (1,0) to[bend right=60] (2,0);

    \draw[arc, red] (4,0) to[bend right=60] (5,0);

    \draw[arc, red] (7,0) to[bend right=60] (8,0);

    \foreach \x in {1,...,8} {
        \node[dot] at (\x,0) {};
    }

\end{tikzpicture}} \\ \hline
\end{tabular}
\caption{On the left is the form of a matrix in $\mathfrak{p}_n\frac{5|3}{2|1|2|1|2}$ and on the right is the meander $\frac{5|3}{2|1|2|1|2}$.}
\label{fig: meander}
\end{center}
\end{figure}
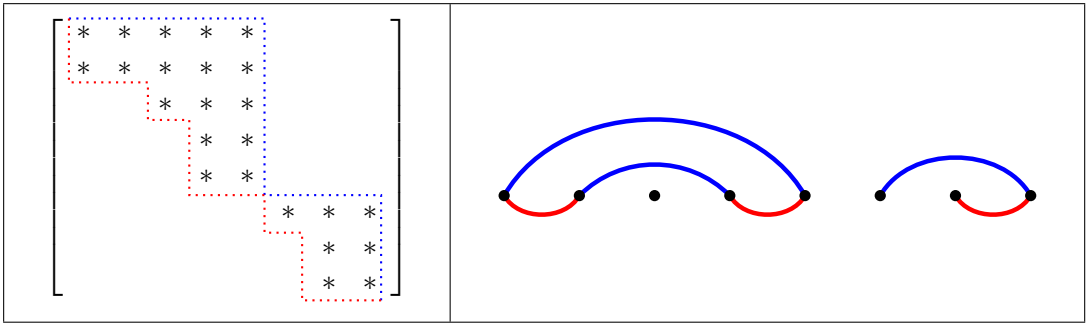

\end{example}

The combinatorial structure of meanders has led to interesting results concerning their index.
In \cite{CHMW15}, the authors provide a recursive classification of meander graphs, showing that each meander is identified by a unique sequence of fundamental graph theoretic moves in a process called \emph{winding down}. Because only one move alters the component structure, the resulting signature  determines the algebraic index.
There have also been various formulas to compute the index based on the greatest common divisor of part sizes in the compositions. For example, in \cite{CGM11} the authors show the index of the meander $\frac{a|b|c}{n}$ is $\gcd(a+b,b+c)$. The question of whether or not a similar formula could be found when the top composition has four or more parts was addressed in \cite{KL17} where the authors found that no such general formula exists. Other formulas for index computation can be found in \cite{CCH20} and \cite{Cet22}. 

While many have focused on the computation of the index based on the combinatorial properties of the meander, others have used these meanders to define statistics on integer partitions (i.e., unordered compositions). In \cite{SY20}, the authors consider the generating function for partitions into odd parts, and in \cite{M23} the author finds partial formulae for the index of partitions whose parts come from a restricted set. 

Our primary motivation in this work is the enumeration of meanders by their index. When the part sizes of the compositions are restricted, we present bivariate generating functions that track composition size $n$ and index $k$.

\begin{table}
    \centering
    \begin{tabular}{c|c||c|c}
        $J$ & Theorem & $J$ & Theorem \\ \hline\hline
        $\{1,2\}$ & \ref{thm: 12Indec} & $\{2,3\}$ & \ref{thm: 23} \\ \hline
        $\{1,3\}$ & \ref{thm: 13Indec} & $\{3,4\}$ & \ref{thm: 34} \\ \hline
        $\{1,4\}$ & \ref{thm: 14Indec} & $\{1,2,3\}$ & \ref{thm: 123} \\ \hline
        $\{1,5\}$ & \ref{thm: 15Indec} & $\{1,3,4\}$ & \ref{thm: 134} \\ \hline
        $\{1,7\}$ & \ref{thm: 17Indec} & $\{1,2,n\}$ & \ref{thm:12n} 
    \end{tabular}
    \caption{For each $J\subseteq[n]$ listed here, the associated theorem presents a bivariate generating function that enumerates meanders by length and index with part sizes restricted to elements from $J$ (or with elements from $cJ$ for any positive integer $c$).}
    \label{tab:summary}
\end{table}

The paper is organized in the following way. Section \ref{sec: structure} explores the structural properties of meanders and their index arising from seaweed algebras with restricted part sizes, with special attention to indecomposable meanders. We then use these properties in Sections \ref{sec: index}, \ref{sec: moreacyclic}, and \ref{sec: cyclic}, where we present our main enumerative results. In Section \ref{sec: index}, we enumerate meanders by length (that is, number of nodes) and index when composition part sizes are restricted to the set $\{1,j\}$ for certain values of $j$ for which the meanders are acyclic. In Section \ref{sec: moreacyclic}, we enumerate meanders by length and index where parts are restricted to other sets $J$ associated to acyclic meanders. In Section \ref{sec: cyclic}, we present a few results regarding non-acyclic meanders. Table \ref{tab:summary} provides a quick reference  for our primary enumerative theorems from these three sections. In Section \ref{sec: bijection}, we define a bijection between indecomposable meanders with restricted part sizes and permutations with restricted displacements. We then investigate how the index of a meander translates to a particular statistic of the associated permutation under this bijection. Section \ref{sec: conclusion} concludes the paper and discusses potential areas for future research and related open questions.

\subsection{Notation}

Throughout this paper, we use Greek letters to denote integer compositions. We say the ordered list $\lambda = (\lambda_1,\ldots,\lambda_r)$ is a composition of $n$ if $\lambda_1+\cdots +\lambda_r=n$. In this case, we write $\lambda\vdash n.$ For the ease of notation, we will often write $\lambda$ as  $\lambda_1|\lambda_2|\cdots|\lambda_r$ instead of as an ordered list. Given two compositions $\lambda\vdash n$ and $\mu\vdash m$, we denote the concatenation of these compositions by $\lambda|\mu$. For example, if $\lambda = 3|1|4$ and $\mu = 6|1|3$, then $\lambda|\mu = 3|1|4|6|1|3.$ Additionally, if a particular element $k$ is repeated $\ell$ times, we may write $k^\ell$ to denote this list of elements. For example, instead of writing the composition $3|3|1|1|1|1|1|1|2,$ we may simply write $3^2|1^6|2.$ We denote the set of partial sums of $\lambda$ by $PS(\lambda).$ For example, if $\lambda = 3|1|4|6|1|1$, then $PS(\lambda) = \{3,4,8,14,15,16\}.$

We denote by $\M_n$ the set of all meanders with $n$ nodes.  Note that $|\M_n| = 4^{n-1}$. For example, 
\[\M_1 = \{\tfrac{1}{1}\} \quad \text{ and } \quad \M_2=\{\tfrac{1|1}{1|1},\tfrac{2}{1|1},\tfrac{1|1}{2},\tfrac{2}{2}\}. \]
Recall that we define $\ind(\frac{\lambda}{\mu})$ to be the index of the associated seaweed algebra and that it can be computed as twice the number of cycles plus the number of paths in the meander graph.
We let $a_{n,k}$ be the number of meanders in $\M_n$ with index $k$, and let $A(x,y)=\sum_{n,k\geq 0} a_{n,k}x^ny^k$ be the associated generating function. 
For a fixed set $J$ of positive integers, we use $\mathcal{M}^J_n$ to denote the set of meanders $\frac{\lambda}{\mu}\in\mathcal{M}_n$ where $\lambda$ and $\mu$ are composed only of elements from $J$, and let $a_{n,k}^J$ be the number of meanders in $\mathcal{M}^J_n$ with index $k$. The associated generating function is denoted $A_J(x,y)=\sum_{n,k\geq 0} a_{n,k}^Jx^ny^k$. 


 Given two meanders $\frac{\lambda_1}{\mu_1}\in \M_n$ and $\frac{\lambda_2}{\mu_2}\in \M_m$, we define the direct sum $\frac{\lambda_1}{\mu_1}\oplus\frac{\lambda_2}{\mu_2}\in \M_{n+m}$ to be the meander $\frac\lambda\mu = \frac{\lambda_1|\lambda_2}{\mu_1|\mu_2}.$ Notice  it follows immediately from the definition that \[\ind(\tfrac\lambda\mu) = \ind(\tfrac{\lambda_1}{\mu_1})+\ind(\tfrac{\lambda_2}{\mu_2}).\]
We define a \textit{decomposable} meander to be one that is the direct sum of strictly smaller meanders and say a meander is \textit{indecomposable} if it is not decomposable. We denote the set of indecomposable meanders in $\M_n$ by $\mathcal{I}_n$ and denote the set of indecomposable meanders in $\M_n^J$ by $\mathcal{I}^J_n$. We  let $i_{n,k}^J$ be the number of meanders in $\mathcal{I}_n^J$ with index $k$ and denote the associated generating function by $I_J(x,y) = \sum_{n,k\geq 0}i_{n,k}^J x^ny^k.$

\section{Structure of seaweed meanders} \label{sec: structure}

In this section, we establish certain structural properties of meanders and their index values.  Since many of our results rely on the idea of indecomposable meanders, we first enumerate the set of such meanders.

\begin{proposition}
    For $n \geq 2$, $|\mathcal{I}_n| = 3^{n-1}$.
\end{proposition}

\begin{proof}
Note that a meander $\frac{\lambda}{\mu}\in\M_n$ is indecomposable  if and only if $PS(\lambda)\cap PS(\mu)=\{n\}$, that is, if their set of partial sums are disjoint except for their total sum $n$. Thus the integers $1$ to $n-1$ must either be in $PS(\lambda), PS(\mu)$, or in neither set. Each possible distribution of these integers is associated to a unique meander, and the result immediately follows.
 \end{proof}


If a meander is decomposable, we can write it uniquely as a direct sum of indecomposable parts as $\frac\lambda\mu = \bigoplus_{i=1}^\ell \frac{\lambda_i}{\mu_i}.$ In this case, it is clear that the index of the meander $\frac\lambda\mu$ is equal to the sum of the index values of the meanders $\frac{\lambda_i}{\mu_i}$, as we state in the next proposition. 

\begin{proposition} \label{prop: ItoA}
    If we write meander $\frac\lambda\mu$ as a direct sum $\frac\lambda\mu = \bigoplus_{i=1}^\ell \frac{\lambda_i}{\mu_i}$ of its indecomposable parts, then 
    \[
    \ind\bigg(\frac\lambda\mu\bigg) = \sum_{i=1}^\ell \ind\bigg(\frac{\lambda_i}{\mu_i}\bigg).
    \] Furthermore, if  $A_J(x,y)$ is the generating function for all meanders with elements in $J$ and $I_J(x,y)$ is the generating function for the indecomposable meanders with elements in $J$, then \[A_J(x,y) = \frac{1}{1-I_J(x,y)}.\]
\end{proposition}
\begin{proof}
    The first statement follows from the fact that the index of a direct sum is the sum of the index values. The second statement is an immediate consequence of this fact together with the fact that the length of the direct sum is the sum of the lengths of the indecomposable parts. 
\end{proof}

We can also describe an inflation of a meander. We say the \textit{inflation} of meander $\frac\lambda\mu\in\M_n$ by a factor of $c$ is the meander $\frac{c\lambda}{c\mu}\in\M_{cn}$ where $c\lambda = c\lambda_1|c\lambda_2|\cdots|c\lambda_r$ and $c\mu = c\mu_1|c\mu_2|\cdots|c\mu_s.$ For example, if we have $\frac\lambda\mu=\frac{3|1|4|4|1}{4|1|1|3|3|1}\in \M_{13}$, then the inflation of this meander by a factor of 2 would be $\frac{2\lambda}{2\mu}=\frac{6|2|8|8|2}{8|2|2|6|6|2}\in\M_{26}.$ Note that if $\frac\lambda\mu$ only has parts in the set $J =\{j_1,\ldots, j_t\}$, then $\frac{c\lambda}{c\mu}$ only has parts in the set $cJ =\{cj_1,\ldots, cj_t\}$.

\begin{proposition} \label{prop: scaleIndex}
    If $\ind(\frac\lambda\mu)=k$, then $\ind(\frac{c\lambda}{c\mu}) = ck.$
\end{proposition}
\begin{proof}
    Let us note that each (upper or lower) arc in the meander $\frac\lambda\mu$ from $i_1$ to $i_2$ with $i_1<i_2$ becomes $c$ nested arcs from $\{c(i_1-1)+1,\ldots,ci_1\}$ to $\{c(i_2-1)+1,\ldots,ci_2\}$. Furthermore, each point $i$ that is not part of an arc (on either the upper or lower part) becomes a nest of $\lfloor\frac{c}{2}\rfloor$ arcs in positions $\{c(i-1)+1,\ldots,ci\}$ (together with an isolated point in the middle if $c$ is odd). As a result, any isolated point becomes a nested set of cycles (and an isolated point if $c$ is odd), each path becomes a nested set of cycles (and one path if $c$ is odd), and each cycle becomes a nested set of cycles. The result follows from this observation.
\end{proof}

An example of this proposition can be observed in Figure~\ref{fig:inflation}.

\begin{figure}[H]
    \centering
        \begin{tikzpicture}[scale=1]
       \draw[line width=.5mm, red](0,0) to[bend left=45] (4,0);
       \draw[line width=.5mm, red](1,0) to[bend left=45] (3,0);
       \draw[line width=.5mm,blue](5,0) to[bend left=45] (6,0);
       \draw[line width=.5mm,red]
       (0,0) to[bend right=45] (1,0);
       \draw[line width=.5mm,blue](2,0) to[bend right=45] (5,0);
       \draw[line width=.5mm,red](3,0) to[bend right=45] (4,0);
        \foreach \x in {0,...,7}{
       \draw[circle,fill] (\x,0)circle[radius=1mm];
    }
    \end{tikzpicture}
    
            \begin{tikzpicture}[scale=.5]
    
       \draw[line width=.5mm, red](0,0) to[bend left=45] (9,0);
       \draw[line width=.5mm, red](1,0) to[bend left=45] (8,0);
       \draw[line width=.5mm, red](2,0) to[bend left=45] (7,0);
       \draw[line width=.5mm, red](3,0) to[bend left=45] (6,0);
       \draw[line width=.5mm, blue](4,0) to[bend left=45] (5,0);
       \draw[line width=.5mm,blue](10,0) to[bend left=45] (13,0);
       \draw[line width=.5mm,blue](11,0) to[bend left=45] (12,0);
       \draw[line width=.5mm,black](14,0) to[bend left=45] (15,0);
       \draw[line width=.5mm,red]
       (0,0) to[bend right=45] (3,0);
       \draw[line width=.5mm,red]
       (1,0) to[bend right=45] (2,0);
       \draw[line width=.5mm,blue](4,0) to[bend right=45] (11,0);
       \draw[line width=.5mm,blue]
       (5,0) to[bend right=45] (10,0);
       \draw[line width=.5mm,red]
       (6,0) to[bend right=45] (9,0);
       \draw[line width=.5mm,red]
       (7,0) to[bend right=45] (8,0);
       \draw[line width=.5mm,blue](12,0) to[bend right=45] (13,0);
       \draw[line width=.5mm,black](14,0) to[bend right=45] (15,0);
       \foreach \x in {0,...,15}{
       \draw[circle,fill] (\x,0)circle[radius=1mm];
    }
    \end{tikzpicture}
    
    \begin{tikzpicture}[scale=.33]

       \draw[line width=.5mm, red](0,0) to[bend left=45] (14,0);
       \draw[line width=.5mm, red](1,0) to[bend left=45] (13,0);
       \draw[line width=.5mm, red](2,0) to[bend left=45] (12,0);
       \draw[line width=.5mm, red](3,0) to[bend left=45] (11,0);
       \draw[line width=.5mm, red](4,0) to[bend left=45] (10,0);
       \draw[line width=.5mm, red](5,0) to[bend left=45] (9,0);
       \draw[line width=.5mm, blue](6,0) to[bend left=45] (8,0);
       \draw[line width=.5mm,blue](15,0) to[bend left=45] (20,0);
       \draw[line width=.5mm,blue](16,0) to[bend left=45] (19,0);
       \draw[line width=.5mm,blue](17,0) to[bend left=45] (18,0);
       \draw[line width=.5mm,black](21,0) to[bend left=45] (23,0);
       \draw[line width=.5mm,red]
       (0,0) to[bend right=45] (5,0);
       \draw[line width=.5mm,red]
       (1,0) to[bend right=45] (4,0);
       \draw[line width=.5mm,red]
       (2,0) to[bend right=45] (3,0);
       \draw[line width=.5mm,blue](6,0) to[bend right=45] (17,0);
       \draw[line width=.5mm,blue]
       (7,0) to[bend right=45] (16,0);
       \draw[line width=.5mm,blue]
       (8,0) to[bend right=45] (15,0);
       \draw[line width=.5mm,red]
       (9,0) to[bend right=45] (14,0);
       \draw[line width=.5mm,red]
       (10,0) to[bend right=45] (13,0);
       \draw[line width=.5mm,red]
       (11,0) to[bend right=45] (12,0);
       \draw[line width=.5mm,blue](18,0) to[bend right=45] (20,0);
       \draw[line width=.5mm,black](21,0) to[bend right=45] (23,0);
       \foreach \x in {0,...,23}{
       \draw[circle,fill] (\x,0)circle[radius=1mm];
    }
    \end{tikzpicture}
    \caption{The meanders $\frac\lambda\mu, \frac{2\lambda}{2\mu},$ and $\frac{3\lambda}{3\mu}$ with $\frac\lambda\mu=\frac{5|2|1}{2|4|1|1}.$ Note that these meanders have indices 4, 8, and 12, respectively.}
    \label{fig:inflation}
\end{figure}
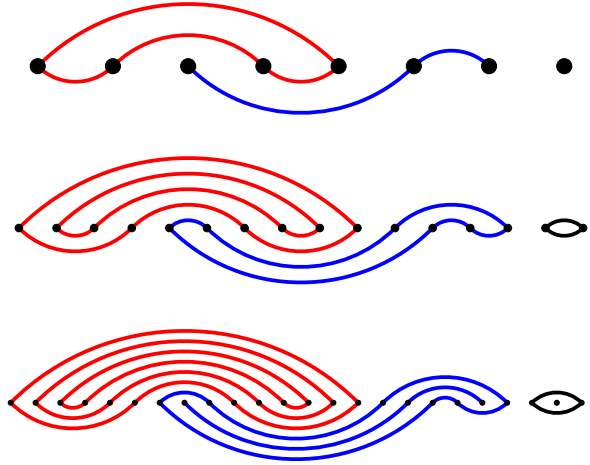


For any meander of the form $\frac{\lambda}{\mu|1^m}\in\M_n$ (that is, any meander where the lower composition ends with at least $m$ 1s), we define the \emph{$m$-connectivity word} to be the word $s_ms_{m-1}\ldots s_2s_1$ where $s_1=1,$ and for $i>1$, $s_i=s_k$ if the $(n-i+1)$st vertex and $(n-k+1)$st vertex are part of the same path or cycle and $s_i = \max\{s_j: j<i\} +1$ otherwise. We use $N_j$ to refer to the $j$-connectivity word corresponding to $\frac{j}{1^j}$ since it will occur often. If $j = 2k$, then $N_{2k} = 1 2\cdots(k-1)kk(k-1)\cdots21$; and if $j = 2k+1$ is odd, then $N_{2k+1} = 12\cdots k(k+1)k\cdots21$.

For example, if we take $\lambda = 3|2|6|1|3$ and $\mu = 3|5$, then $m = 15-8 = 7$. The $6$-connectivity word of $\frac{3|2|6|1|3}{3|5|1|1|1|1|1|1|1}$ is $543121$.

\begin{theorem} \label{thm: shiftStart}
    Suppose that for $i\in \{1,2\}$, we have $\lambda_i\vdash k_i$, $\mu_i\vdash \ell_i$ with $\ell_i+m=k_i$ for some $m\ge1$. If $\frac{\lambda_1}{\mu_1|1^m}$ and $\frac{\lambda_2}{\mu_2|1^m}$ have the same $m$-connectivity word, then for any compositions $\alpha \vdash a, \beta \vdash b$ with $a=b-m$, we have 
    \[
    \ind\left(\frac{\lambda_1|\alpha}{\mu_1|\beta}\right) =\ind\left(\frac{\lambda_2|\alpha}{\mu_2|\beta}\right) + \epsilon
    \]
    where $\epsilon = \ind\left(\frac{\lambda_1}{\mu_1|1^m}\right)-\ind\left(\frac{\lambda_2}{\mu_2|1^m}\right)$.
\end{theorem}

\begin{proof}
    It is enough to show that 
    \[\ind\left(\frac{\lambda_1|\alpha}{\mu_1|\beta}\right) - \ind\left(\frac{\lambda_1}{\mu_1|1^m}\right)=\ind\left(\frac{\lambda_2|\alpha}{\mu_2|\beta}\right) -\ind\left(\frac{\lambda_2}{\mu_2|1^m}\right), \]
    i.e., that the change in index by appending $\alpha$ to $\lambda_i$ and $\beta$ to $\mu_i$ is the same for each $i\in\{1,2\}$. However, this follows from the fact that $\frac{\lambda_1}{\mu_1|1^m}$ and $\frac{\lambda_2}{\mu_2|1^m}$ have the same $m$-connectivity word. Indeed, since the addition of $\alpha$ and $\beta$ will create paths, cycles, or lengthen existing paths in the first meander if and only if it does in the second meander, the result must follow.
\end{proof}

We will often use this result to produce recursions by shortening the start of a meander, we now consider an example.

\begin{example}
    Consider the case where $\lambda_1 = 4|3$, $\mu_1 = 3|3$, $\lambda_2 = 4$, and $\mu_2 = 3$. Notice both $\frac{4|3}{3|3|1}$ and $\frac{4}{3|1}$ have the same $1$-connectivity word. The index of $\frac{4|3}{3|3|1}$ is $2$ and the index of $\frac{4}{3|1}$ is $1$. Theorem \ref{thm: shiftStart} implies that the number of meanders of length $n$ and index $k$ beginning with $\frac{4|3}{3|3}$ is equal to the number of meanders of length $n-3$ and index $k-1$ beginning with $\frac{4}{3}$. So if we consider the meanders,
    \begin{center}
\begin{tikzpicture}[
    dot/.style={circle, fill, inner sep=1.5pt},
    arc/.style={line width=0.6mm}
]

    \draw[arc, blue] (0,0) to[bend left=45] (3,0);
    \draw[arc, blue] (1,0) to[bend left=45] (2,0);
     \draw[arc, blue] (4,0) to[bend left=45] (6,0);
\draw[arc, blue] (0,0) to[bend right=45] (2,0);
    \draw[arc, blue] (3,0) to[bend right=45] (5,0);
    \draw[arc, red,dashed] (6,0) to[bend right=45] (7,-1/3);

     \node[dot] at (0,0) {};
    \node[dot] at (1,0) {};
    \node[dot] at (2,0) {};
    \node[dot] at (3,0) {};
    \node[dot] at (4,0) {};
    \node[dot] at (5,0) {};
    \node[dot] at (6,0) {};

    \draw[arc, blue] (9,0) to[bend left=45] (12,0);
    \draw[arc, blue] (10,0) to[bend left=45] (11,0);
   \draw[arc, blue] (9,0) to[bend right=45] (11,0);
    \draw[arc, red,dashed] (12,0) to[bend right=45] (13,-1/3);

        \node[dot] at (9,0) {};
    \node[dot] at (10,0) {};
     \node[dot] at (11,0) {};
    \node[dot] at (12,0) {};

\end{tikzpicture}
\end{center}
then we shorten the left meander by replacing it with the right meander. The first part of $\beta$ is the red dashed line. Notice we've ``lost'' one path by doing this replacement, thus reducing the index by~1.
\end{example}

\begin{remark}
    Theorem \ref{thm: shiftStart} is sometimes equivalent to the notion of ``winding down'' from \cite{CHMW15}, but not always. If you wound down the meander in the previous example, you would end up with $\frac{1|3}{3|1}$ instead of $\frac{4}{3|1}$. The meander $\frac{1|3}{3|1}$ also has the same $1$-connectivity word and index as $\frac{4|3}{3|3|1}$, so our result accommodates these winding down moves. The deterministic nature of winding down is useful for encoding a meander and its index, but we need the flexibility of Theorem \ref{thm: shiftStart} to establish many of the recursions later in this paper.
\end{remark}


\section{Index of meanders on acyclic sets $\{1,j\}$}\label{sec: index}

We say that a meander is {\em acyclic} if it does not contain any cycles.  For a finite index set $J$, we say that the set $J$ is \emph{acyclic} if every meander in $\mathcal{I}_n^J$ is acyclic for all $n > \max(J)$. In this section, we will find a family of recursions for acyclic meanders of the form $J = \{1,j\}$.

\begin{proposition} \label{prop: 1jAcyclic}
   The set $\{1,j\}$ is acyclic if and only if $j \in \{2,3,4,5,7\}$.
\end{proposition}

\begin{proof}
To prove that indecomposable meanders on $\{1,j\}$ are all acyclic if $j\in\{2,3,4,5,7\}$, we proceed by cases. In the case when $j=2$, clearly a cycle can only appear if the arc from an occurrence of 2 on the upper half of the meander lines up with an arc corresponding to another occurrence of 2 on the lower half of the meander. However, this contradicts that the meander was indecomposable. An identical argument works for when $j=3$. For the remaining cases, we exclude the possibility of the cycle originating from the largest arc of an occurrence of $j$ since this would contradict indecomposability in a similar way.

In the case when $j=4,$ any cycle must originate from the smaller inner arc of an occurrence of 4. However, 
as we can see in the image below, it is impossible to ``cap'' the other end of that potential cycle, so a cycle is not possible. Similarly with $j=5,$ also pictured below.

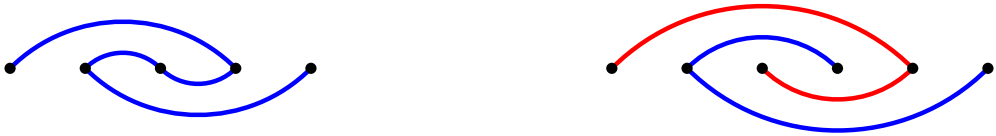
\begin{center}
\begin{tikzpicture}[
    dot/.style={circle, fill, inner sep=1.5pt},
    arc/.style={line width=0.6mm}
]

    \draw[arc, blue] (0,0) to[bend left=45] (3,0);
    \draw[arc, blue] (1,0) to[bend left=45] (2,0);
\draw[arc, blue] (1,0) to[bend right=45] (4,0);
    \draw[arc, blue] (2,0) to[bend right=45] (3,0);

     \node[dot] at (0,0) {};
    \node[dot] at (1,0) {};
    \node[dot] at (2,0) {};
    \node[dot] at (3,0) {};
    \node[dot] at (4,0) {};

    \draw[arc, red] (8,0) to[bend left=45] (12,0);
    \draw[arc, blue] (9,0) to[bend left=45] (11,0);
   \draw[arc, blue] (9,0) to[bend right=45] (13,0);
    \draw[arc, red] (10,0) to[bend right=45] (12,0);

        \node[dot] at (8,0) {};
    \node[dot] at (9,0) {};
     \node[dot] at (10,0) {};
    \node[dot] at (11,0) {};
    \node[dot] at (12,0) {};
      \node[dot] at (13,0) {};

\end{tikzpicture}
\end{center}

Finally, let us consider the case when $j=7.$ In this case, there are two inner arcs that could initiate the start of a cycle. However, as we can see below, it is impossible to cap these potential cycles. 

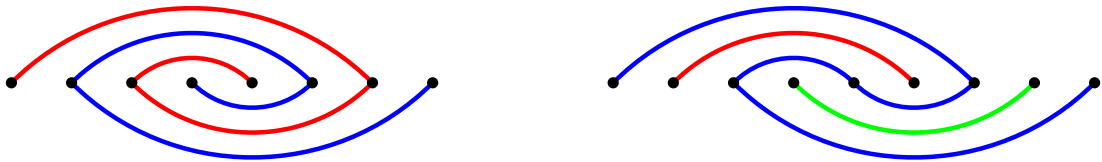
\begin{center}
\begin{tikzpicture}[ scale =.8,
    dot/.style={circle, fill, inner sep=1.5pt},
    arc/.style={line width=0.6mm}
]

    \draw[arc, red] (0,0) to[bend left=45] (6,0);
    \draw[arc, blue] (1,0) to[bend left=45] (5,0);
    \draw[arc, red] (2,0) to[bend left=45] (4,0);
\draw[arc, blue] (1,0) to[bend right=45] (7,0);
    \draw[arc, red] (2,0) to[bend right=45] (6,0);
    \draw[arc, blue] (3,0) to[bend right=45] (5,0);

    \node[dot] at (0,0) {};
    \node[dot] at (1,0) {};
    \node[dot] at (2,0) {};
    \node[dot] at (3,0) {};
    \node[dot] at (4,0) {};
    \node[dot] at (5,0) {};
    \node[dot] at (6,0) {};
    \node[dot] at (7,0) {};

    \draw[arc, blue] (10,0) to[bend left=45] (16,0);
    \draw[arc, red] (11,0) to[bend left=45] (15,0);
    \draw[arc, blue] (12,0) to[bend left=45] (14,0);
   \draw[arc, blue] (12,0) to[bend right=45] (18,0);
    \draw[arc, green] (13,0) to[bend right=45] (17,0);
    \draw[arc, blue] (14,0) to[bend right=45] (16,0);

    \node[dot] at (10,0) {};
    \node[dot] at (11,0) {};
     \node[dot] at (12,0) {};
    \node[dot] at (13,0) {};
    \node[dot] at (14,0) {};
      \node[dot] at (15,0) {};
      \node[dot] at (16,0) {};
      \node[dot] at (17,0) {};
      \node[dot] at (18,0) {};

\end{tikzpicture}
\end{center}
Now, let us see that in every other case, we do have examples of indecomposable meanders with cycles.
If $j=2k$ is even with $j\geq 6,$ then the meander $\frac{j|1^{k-3} | j | 1^{k-1}}{1^{k-1}|j|1^{k-3}|j}$ contains a cycle. If $j=2k+1$ is odd with $j\geq 9,$ then the meander $\frac{j|1^{k-4} | j | 1^{k-1}}{1^{k-1}|j|1^{k-4}|j}$ contains a cycle.
As an example, the case when $j=6$ is shown below.
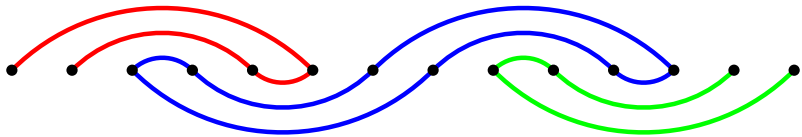
\begin{center}
\begin{tikzpicture}[scale=.8,
    dot/.style={circle, fill, inner sep=1.5pt},
    arc/.style={line width=0.6mm}
]

    \draw[arc, red] (0,0) to[bend left=45] (5,0);
    \draw[arc, red] (1,0) to[bend left=45] (4,0);
      \draw[arc, blue] (2,0) to[bend left=45] (3,0);
\draw[arc, blue] (2,0) to[bend right=45] (7,0);
    \draw[arc, blue] (3,0) to[bend right=45] (6,0);
     \draw[arc, red] (4,0) to[bend right=45] (5,0);
    \draw[arc, blue] (6,0) to[bend left=45] (11,0);
    \draw[arc, blue] (7,0) to[bend left=45] (10,0);
      \draw[arc, green] (8,0) to[bend left=45] (9,0);
\draw[arc, green] (8,0) to[bend right=45] (13,0);
    \draw[arc, green] (9,0) to[bend right=45] (12,0);
     \draw[arc, blue] (10,0) to[bend right=45] (11,0);

     \node[dot] at (0,0) {};
    \node[dot] at (1,0) {};
    \node[dot] at (2,0) {};
    \node[dot] at (3,0) {};
    \node[dot] at (4,0) {};
     \node[dot] at (5,0) {};
      \node[dot] at (6,0) {};
       \node[dot] at (7,0) {};
       \node[dot] at (8,0) {};
       \node[dot] at (9,0) {};
       \node[dot] at (10,0) {};
       \node[dot] at (11,0) {};
        \node[dot] at (12,0) {};
         \node[dot] at (13,0) {};

\end{tikzpicture}
\end{center}
\end{proof}

We will enumerate meanders of length $n$ with index $k$ where the compositions defining the meander have parts from an acyclic set $\{1,j\}$, that is, when $j\in\{2,3,4,5,7\}$. Let us first consider the simplest case, when $J = \{1,2\}$.

\begin{theorem} \label{thm: 12Indec}
    We have
    \[
     A_{\{1,2\}}(x,y) = \frac{1-x}{(1-xy-x^2y^2)(1-x) - 2x^2y}
    \]
    and for any integer $c \geq 1$, $A_{\{c,2c\}}(x,y) = A_{\{1,2\}}(x^c,y^c)$.
\end{theorem}

\begin{proof}
First note that for $n>2,$ there are only two indecomposable meanders, namely $\frac{1|2|1|2|\cdots}{2|1|2|1|\cdots}$ and $\frac{2|1|2|1|\cdots}{1|2|1|2\cdots}$. Furthermore, each of these has index 1. Taking into consideration the indecomposable meanders of size $n\leq 2,$ we thus have
\[
I_{\{1,2\}}(x,y) = xy +x^2y^2 + \frac{2x^2y}{1-x}.
\]
Using Proposition \ref{prop: ItoA}, we know $A_{\{1,2\}}(x,y)=\frac{1}{1-I_{\{1,2\}}(x,y)},$ from which the statement of the theorem follows. The statement regarding $A_{\{c,2c\}}(x,y)$ follows directly from Proposition \ref{prop: scaleIndex}. 
\end{proof}

To enumerate meanders with a given index whose elements come from the acyclic set $\{1,j\}$ for $j>2,$ we utilize the following lemma. We first define the following notation when $n\geq j$:
    \begin{itemize}
        \item let $b_{n,k}^\ell(\{1,j\})$ be the number of $\frac{\lambda}{\mu}\in\mathcal{I}^{\{1,j\}}_{n,k}$ of the form $\frac{j|\alpha}{1^\ell| \beta}$ or $\frac{1^\ell|\beta}{j| \alpha}$ for some compositions $\alpha\vdash n-j$ and $\beta\vdash n-\ell$, and 
        \item let $c_{n,k}^\ell(\{1,j\})$ be the number of $\frac{\lambda}{\mu}\in\mathcal{I}^{\{1,j\}}_{n,k}$ of the form $\frac{j|\alpha}{1^\ell|j| \beta}$ or $\frac{1^\ell|j|\beta}{j| \alpha}$ for some compositions $\alpha \vdash n-j$ and $\beta\vdash n-j-\ell$.
    \end{itemize}
Using the notation above, the following lemma establishes recursive formulas for $i^{\{1,j\}}_{n,k}$ in terms of $b_{n,k}^\ell(\{1,j\})$ and $c_{n,k}^\ell(\{1,j\})$ in the case where $\{1,j\}$ is acyclic. We note that removing the index $k$, these recursive formulas hold more generally for $i^{\{1,j\}}_{n}$ for any $j$.

\begin{lemma} \label{lem: acyclicRecursive}
    Let $k\geq 1,$ $j\geq 2,$ and $n \geq 2j + \lfloor j/2 \rfloor$. Additionally, let $b_{n,k}^\ell:=b_{n,k}^\ell(\{1,j\})$ and $c_{n,k}^\ell:=c_{n,k}^\ell(\{1,j\})$.
    If $\{1,j\}$ is acyclic and $\delta = j \pmod 2$, then the following recursions hold:
    \begin{enumerate}
        \item $i^{\{1,j\}}_{n,k} = b^1_{n,k} =\sum_{\ell=1}^{j-1} c^\ell_{n,k}$.
        \item For $\ell > 1$, $b^\ell_{n,k} = b^{\ell-1}_{n,k} - c^{\ell-1}_{n,k}$
        \item For $\ell \geq \lceil j/2 \rceil$, $c^\ell_{n,k} = b^{j-\ell}_{n-\ell,k-\ell + \lfloor j/2 \rfloor}$
        \item For $\ell \le \lfloor j/2 \rfloor$, $c^\ell_{n,k} = \sum_{i=1}^\ell b^i_{n-j-\ell+i,k-\ell+i-\delta}$.

    \end{enumerate}
    Disregarding the index, the recursions hold for any set $\{1,j\}$.
\end{lemma}

Before we prove Lemma \ref{lem: acyclicRecursive}, we include an example to illustrate. 

\begin{example}
    To illustrate how the fourth recursion in Lemma \ref{lem: acyclicRecursive} works consider the following example. Let $j = 7$, $\ell = 2$. Consider a meander that begins as below:

    \begin{center}
    \begin{tikzpicture}[scale=.8,
    dot/.style={circle, fill, inner sep=1.5pt},
    arc/.style={line width=0.6mm}
]

    \draw[arc] (0,0) to[bend left=45] (6,0);
    \draw[arc] (1,0) to[bend left=45] (5,0);
      \draw[arc] (2,0) to[bend left=45] (4,0);
\draw[arc] (2,0) to[bend right=45] (8,0);
    \draw[arc] (3,0) to[bend right=45] (7,0);
     \draw[arc] (4,0) to[bend right=45] (6,0);
    

     \node[dot] at (0,0) {};
    \node[dot] at (1,0) {};
    \node[dot] at (2,0) {};
    \node[dot] at (3,0) {};
    \node[dot] at (4,0) {};
     \node[dot] at (5,0) {};
      \node[dot] at (6,0) {};
       \node[dot] at (7,0) {};
       \node[dot] at (8,0) {};
\end{tikzpicture}
\end{center}
This is an example of a meander enumerated by $c_{n,k}^2$, since it begins with $\frac{7}{1|1|7}$. Since $\ell = 2 < 3 = \lfloor 7/2 \rfloor$ we must use recursion $4$. The idea for this recursion is that an indecomposable meander beginning with $\frac{7}{1|1|7}$ could either have its top composition continue as $\frac{7|7}{1|1|7}$ or $\frac{7|1|7}{1|1|7}$. In the first case, the meander looks like:
  \begin{center}
    \begin{tikzpicture}[scale=.8,
    dot/.style={circle, fill, inner sep=1.5pt},
    arc/.style={line width=0.6mm}
]

    \draw[arc, color = red] (0,0) to[bend left=45] (6,0);
    \draw[arc, color = red] (1,0) to[bend left=45] (5,0);
      \draw[arc, color = red] (2,0) to[bend left=45] (4,0);
       \draw[arc] (7,0) to[bend left=45] (13,0);
    \draw[arc] (8,0) to[bend left=45] (12,0);
      \draw[arc] (9,0) to[bend left=45] (11,0);
\draw[arc, color = red] (2,0) to[bend right=45] (8,0);
    \draw[arc, color = red] (3,0) to[bend right=45] (7,0);
     \draw[arc, color = red] (4,0) to[bend right=45] (6,0);
    

     \node[dot] at (0,0) {};
    \node[dot] at (1,0) {};
    \node[dot] at (2,0) {};
    \node[dot] at (3,0) {};
    \node[dot] at (4,0) {};
     \node[dot] at (5,0) {};
      \node[dot] at (6,0) {};
       \node[dot] at (7,0) {};
       \node[dot] at (8,0) {};
       \node[dot] at (9,0) {};
       \node[dot] at (10,0) {};
       \node[dot] at (11,0) {};
        \node[dot] at (12,0) {};
         \node[dot] at (13,0) {};
\end{tikzpicture}
\end{center}
in which case, we remove the red arcs and can instead enumerate the meanders beginning with $\frac{7}{1|1}$ of length $n-7$ and index $k - 1$, which is $b_{n-7,k-1}^2$. 

In the second case, the meander looks like:
    \begin{center}
    \begin{tikzpicture}[scale=.8,
    dot/.style={circle, fill, inner sep=1.5pt},
    arc/.style={line width=0.6mm}
]

    \draw[arc, color = red] (0,0) to[bend left=45] (6,0);
    \draw[arc, color = red] (1,0) to[bend left=45] (5,0);
      \draw[arc, color = red] (2,0) to[bend left=45] (4,0);
       \draw[arc] (8,0) to[bend left=45] (14,0);
    \draw[arc] (9,0) to[bend left=45] (13,0);
      \draw[arc] (10,0) to[bend left=45] (12,0);
\draw[arc, color = red] (2,0) to[bend right=45] (8,0);
    \draw[arc, color = red] (3,0) to[bend right=45] (7,0);
     \draw[arc, color = red] (4,0) to[bend right=45] (6,0);
    

     \node[dot] at (0,0) {};
    \node[dot] at (1,0) {};
    \node[dot] at (2,0) {};
    \node[dot] at (3,0) {};
    \node[dot] at (4,0) {};
     \node[dot] at (5,0) {};
      \node[dot] at (6,0) {};
       \node[dot] at (7,0) {};
       \node[dot] at (8,0) {};
       \node[dot] at (9,0) {};
       \node[dot] at (10,0) {};
       \node[dot] at (11,0) {};
        \node[dot] at (12,0) {};
         \node[dot] at (13,0) {};
         \node[dot] at (14,0) {};
\end{tikzpicture}
\end{center}
in which case, we remove the red arcs and can instead enumerate the meanders beginning with $\frac{7}{1}$ of length $n-8$ and index $k - 2$, which is $b_{n-8,k-2}^1$. This shows that $c_{n,k}^2 = b_{n-7,k-1}^2 + b_{n-8,k-2}^1$ for $n \geq 17$. The reduction in index relies on the acyclic assumption. We now prove Lemma \ref{lem: acyclicRecursive}.

\end{example}

\begin{proof}
Let $\frac{\lambda}{\mu}\in\mathcal{I}_{n,k}^{\{1,j\}}.$
Without loss of generality, suppose $\lambda_1>\mu_1$, i.e., that  $\lambda_1=j$ and $\mu_1=1$.  

    The first recursive formula follows immediately from definitions since an indecomposable meander could have $\mu$ begin with $1^i$ for $1 \leq i < j$, before the occurrence of the first $j$. The second recursive formula also follows immediately since meanders that have $\mu$ begin with $1^{\ell-1}$  could either have $j$ or $1$ as their next entry in $\mu$. 

    For the third recursion, we use Theorem \ref{thm: shiftStart}. Essentially, we will start with $\frac{\lambda}{\mu} = \frac{j|\alpha}{1^\ell |j| \beta}$, replace it with $\frac{\lambda'}{\mu'} = \frac{1^{j-\ell}|\alpha}{j|\beta}$, and track how the index has changed using the theorem. Notice first that for $\ell \geq \lceil j/2 \rceil$, the $\ell$-connectivity word of both $\frac{1^\ell| j}{j|1^\ell}$ and $\frac{j}{1^{j-\ell}| 1^{\ell}}$ is $N_{2\ell-j} \ominus ((j-\ell)\cdots 21)$. Furthermore, the index of $\frac{1^\ell| j}{j|1^\ell}$ is $2 \lceil j/2 \rceil + \ell - j$, and the index of $\frac{j}{1^{j-\ell}|1^\ell}$ is $\lceil j/2 \rceil$. Theorem \ref{thm: shiftStart} then implies that number of meanders of length $n$ with index $k$ beginning with $\frac{1^\ell| j}{j}$ is equal to the number of meanders of length $n-\ell$ with index \[k - (2 \lceil j/2 \rceil + \ell - j) + \lceil j/2 \rceil = k - \ell + \lfloor j/2 \rfloor\] beginning with $\frac{j}{1^{j-\ell}}$.  

    For the final recursion, we make a similar argument. Notice for $1 \leq i \leq \ell \le \lfloor j/2\rfloor$ and $\{1,j\}$ acyclic, the $(j-i)$-connectivity word for $\frac{j|1^{\ell-i}|j}{1^\ell| j|1^{j-i}}$ matches the $(j-i)$-connectivity word for $\frac{j}{1^{i}|1^{j-i}}$. Indeed, they are both equal to $N_{j-2i} \ominus (i\cdots21)$. And for $\ell \le \lfloor j/2 \rfloor$, the index of $\frac{j|1^{\ell-i}|j}{1^\ell |j |1^{j-i}}$ is $\lceil j/2 \rceil + \delta +\ell - i$, and the index of $\frac{j}{1^{i}| 1^{j-i}}$ is $\lceil j/2 \rceil$. 

    As a result, by Theorem \ref{thm: shiftStart} the number of meanders of length $n$ with index $k$ of the form $\frac{j|1^{\ell-i}|j|\alpha}{1^\ell| j|\beta}$ is equal to the number of meanders of length $n-j-\ell+i$ with index $k-\ell+i-\delta$ of the form $\frac{j|\alpha}{1^{i}|\beta}$, which are precisely enumerated by $b_{n-j-\ell+i,k-\ell+i -\delta}^i$. But $c_{n,k}^\ell$ is the sum of all the meanders of length $n$ and index $k$ of the form $\frac{j|1^{\ell-i}|j|\alpha}{1^\ell| j|\beta}$ for $1 \leq i \leq \ell$, which therefore equals $\sum_{i=1}^\ell b_{n-j-\ell+i,k-\ell+i -\delta}^i$ as desired.

    If we do not have to track the index, this means the connectivity word and index are irrelevant. So all the reductions based on how the meander started that we make in the previous paragraphs hold generally.
\end{proof}

\begin{remark}
    Notice that only the final recursion relies on the acyclic assumption to ensure that the connectivity words are the same when we make our reduction. All of the others hold generally.
\end{remark}

Using Lemma \ref{lem: acyclicRecursive} in combination with Proposition \ref{prop: 1jAcyclic}, we can find recursions and generating functions for $i^{\{1,j\}}_{n,k}$ with $j \in \{3,4,5,7\}$. (Note this lemma could also have been used to prove Theorem~\ref{thm: 12Indec}.) This then allows us to find $A_{\{c,jc\}}(x,y)$ for all $c \geq 1$ and $j \in \{3,4,5,7\}$ by applying Proposition \ref{prop: ItoA} and Proposition \ref{prop: scaleIndex}.

\begin{theorem} \label{thm: 13Indec}
    We have
    \[
        A_{\{1,3\}}(x,y) = \frac{1-x^2y - x^3y}{(1-x^2y - x^3y)(1-xy-x^3y^{3})-2x^3y^2(1+x)}
    \]
    and for any $c \geq 1$, $A_{\{c,3c\}}(x,y) = A_{\{1,3\}}(x^c,y^c)$.
\end{theorem}

\begin{proof}
    For the duration of this proof, let $i_{n,k}:=i^{\{1,3\}}_{n,k}$, $b_{n,k}^\ell := b_{n,k}^\ell(\{1,3\})$ and $c_{n,k}^\ell := c_{n,k}^\ell (\{1,3\})$. Since $\{1,3\}$ is acyclic by Proposition~\ref{prop: 1jAcyclic}, we can apply the recursions from Lemma~\ref{lem: acyclicRecursive} with $j=3$. The third and fourth part of this lemma imply that for $n\geq 7$, $c^1_{n,k} = b^1_{n-3,k-1} = i_{n-3,k-1}$ and $c^2_{n,k} = b^1_{n-2,k-1} = i_{n-2,k-1}$. Thus, since $i_{n,k} = b^1_{n,k} = c^1_{n,k} + c^2_{n,k},$ we have $i_{n,k} = i_{n-2,k-1} + i_{n-3,k-1}$ for $n\geq 7$, which together with initial conditions, gives us
    \[
    I_{\{1,3\}}(x,y) = xy + x^3y^3 + \frac{2x^3y^2(1+x)}{1-x^{2}y - x^{3}y}
    \]
    Proposition \ref{prop: ItoA} then implies the desired result. We can then find $A_{\{c,3c\}}(x,y)$ for all $c \geq 1$ by Proposition \ref{prop: scaleIndex}.
\end{proof}



\begin{theorem} \label{thm: 14Indec}
    We have 
    \[
        A_{\{1,4\}}(x,y) = \frac{1-x^{2}-x^{4}-x^{3}y+x^{6}}{(1-xy-x^4y^4)(1-x^{2}-x^{4} - x^{3}y+x^{6})-2x^4y(x+y-x^3)}
    \]
    and for any $c \geq 1$, $A_{\{c,4c\}}(x,y) = A_{\{1,4\}}(x^c,y^c)$.
\end{theorem}

\begin{proof}
   For the duration of this proof, let $i_{n,k}:=i^{\{1,4\}}_{n,k}$, $b_{n,k}^\ell := b_{n,k}^\ell(\{1,4\})$ and $c_{n,k}^\ell := c_{n,k}^\ell (\{1,4\})$. Since $\{1,4\}$ is acyclic by Lemma~\ref{prop: 1jAcyclic}, we can apply the recursions from Lemma~\ref{lem: acyclicRecursive} with $j=4$.
   This lemma implies that for $n\ge12$, $c^1_{n,k} = b^1_{n-4,k} = i_{n-4,k}$, \[c^2_{n,k} = b^{2}_{n-2,k}=b^1_{n-2,k} - c^1_{n-2,k} = i_{n-2,k} - i_{n-6,k},\] 
   since $c_{n-2,k}^1 = b_{n-6,k}^1 = i_{n-6,k}$, and  $c^3_{n,k} = b^1_{n-3,k-1} = i_{n-3,k-1}$. 
   Thus for $n\ge 12$ (and by direct computation for $n = 11$),
    \[
    i_{n,k} = i_{n-2,k} + i_{n-3,k-1} + i_{n-4,k} - i_{n-6,k}
    \]
    Accounting for initial conditions, the generating function for the indecomposable meanders is 
    \[
        I_{\{1,4\}}(x,y) = xy + x^4y^4 +  \frac{2x^4y(x+y-x^3)}{1- x^{2}-x^{4}-x^{3}y + x^{6}}.
    \]
    Propositions \ref{prop: ItoA} and \ref{prop: scaleIndex} yield the desired results.
\end{proof}


\begin{theorem} \label{thm: 15Indec}
   We have $A_{\{1,5\}}(x,y)  = \frac{f(x)}{g(x)}$, where
   \begin{align*}
   f(x) &= 1-x^3y-x^4y^2-2x^5y-x^6y^2+x^8y^2+x^{10}y^2\\
   g(x) &= (1-xy-x^5y^5)f(x)-2x^5y^2(x+y+x^2y-x^4y-x^6y)
   \end{align*}
   and for any $c \geq 1$, $A_{\{c,5c\}}(x,y) = A_{\{1,5\}}(x^c,y^c)$.
\end{theorem}

\begin{proof}
    For the duration of this proof, let $i_{n,k}:=i^{\{1,5\}}_{n,k}$, $b_{n,k}^\ell := b_{n,k}^\ell(\{1,5\})$ and $c_{n,k}^\ell := c_{n,k}^\ell (\{1,5\})$. Since $\{1,5\}$ is acyclic by Lemma~\ref{prop: 1jAcyclic}, we can apply the recursions from Lemma~\ref{lem: acyclicRecursive} with $j=5$. We use this lemma to compute $c_{n,k}^\ell$ for $1\leq \ell\leq 4.$ In particular, we can find that for $n \ge 17$ $c_{n,k}^1 = b^1_{n-5,k-1} = i_{n-5,k-1}$, 
    \begin{align*}
    c_{n,k}^2 &= b^2_{n-5,k-1} + b^1_{n-6,k-2}\\
    &= b^1_{n-5,k-1} - c^1_{n-5,k-1} + i_{n-6,k-2}\\
    &= i_{n-5,k-1} - i_{n-10,k-2}+ i_{n-6,k-2},
    \end{align*}
    and that \[ c_{n,k}^3 = b^2_{n-3,k-1}= b^1_{n-3,k-1} - c^1_{n-3,k-1}
        = i_{n-3,k-1} - i_{n-8,k-2}\] and  $c^4_{n,k} = b^1_{n-4,k-2} = i_{n-4,k-2}$. Finally, using the fact that $i_{n,k} = b_{n,k}^1=\sum_{\ell=1}^4 c_{n,k}^\ell,$ we obtain for $n \ge 17$ (also for $n = 16$ checked directly),
       \begin{align*}
        i_{n,k} &= i_{n-3,k-1} + i_{n-4,k-2} + 2 i_{n-5,k-1} + i_{n-6,k-2} - i_{n-8,k-2} - i_{n-10,k-2}.
   \end{align*}
    Accounting for initial conditions, we find $I_{\{1,5\}}(x,y)$ is the rational function,
    \[I_{\{1,5\}}(x,y) = xy + x^5y^5 +
     \frac{2x^5y^2(x+y+x^2y-x^4y-x^6y)}{1-x^3y-x^4y^2-2x^5y-x^6y^2+x^8y^2+x^{10}y^2}
    \]
    Propositions \ref{prop: ItoA} and \ref{prop: scaleIndex} yield the desired results.
\end{proof}

We can apply the same process to find a generating function when $J = \{1,7\}$. We suppress some of the steps because they are identical to the previous results.

\begin{theorem} \label{thm: 17Indec}
    We have $A_{\{1,7\}}(x,y) = \dfrac{f(x)}{g(x)}$, where
    \begin{align*}
    f(x) &=1 - x^4 y- x^5 y^2  - x^6 y^3 - 3 x^7 y - 2 x^8 y^2 - x^9 y^3 - x^{18} y^3 - x^{21} y^3 + 2 x^{11} y^2 \\ &\quad+ 2 x^{12} y^3 + 3 x^{14} y^2 + 2 x^{15} y^3\\
    g(x) &= f(x)(1-xy-x^7y^7) - 2x^7y^4- 2x^8y^2 - 2x^9y^3-2x^{10}y^4+2x^{12}y^3+4x^{13}y^4\\
    &\quad+4x^{15}y^3+4x^{16}y^4-2x^{19}y^4 - 2x^{22}y^4
    \end{align*}
    and for any $c \geq 1$, $A_{\{c,7c\}}(x,y) = A_{\{1,7\}}(x^c,y^c)$.
\end{theorem}

\begin{proof}
    For the remainder of this proof, let $i_{n,k} \coloneqq i^{\{1,7\}}_{n,k}$. By iteratively applying the recursions from Lemma \ref{lem: acyclicRecursive} as in the previous results, we find
    \begin{align*}
    c_{n,k}^1 &= i_{n-7,k-1}, \quad 
    c_{n,k}^2 
    = i_{n-7,k-1} - i_{n-14,k-2} + i_{n-8,k-2},\\
    c_{n,k}^3 
    &= i_{n-7,k-1} + i_{n-8,k-2} + i_{n-9,k-3} - 2i_{n-14,k-2} - 2i_{n-15,k-3} + i_{n-21,k-3},\\
        c_{n,k}^4 
        &= i_{n-4,k-1} - 2i_{n-11,k-2} + i_{n-18,k-3} - i_{n-12,k-3},\\
        c_{n,k}^5 
        &= i_{n-5,k-2} - i_{n-12,k-3}, \quad \text{ and }\quad 
        c_{n,k}^6 = i_{n-6,k-3}.
    \end{align*}
    Taken together, we find that for $n > 28$, justified for $n \ge 31$ by Lemma \ref{lem: acyclicRecursive} and checked directly for $n = 29,30$,
    \begin{align*}
    i_{n,k} &= i_{n-4,k-1} + i_{n-6,k-3} + 3 i_{n-7,k-1} + i_{n-5,k-2} + 2 i_{n-8,k-2} + i_{n-9,k-3} + i_{n-18,k-3} + i_{n-21,k-3}  \\
    & \quad- 2i_{n-11,k-2} -2 i_{n-12,k-3} - 3 i_{n-14,k-2}  - 2 i_{n-15,k-3}.
    \end{align*}
    Accounting for initial conditions, we conclude that
    \[
    I_{\{1,7\}}(x,y) = xy + x^7y^7 + \frac{f(x)}{g(x)}
    \]
    with
    \begin{align*}
    f(x) &= 2x^7y^4+ 2x^8y^2 + 2x^9y^3+2x^{10}y^4-2x^{12}y^3-4x^{13}y^4-4x^{15}y^3-4x^{16}y^4+2x^{19}y^4 + 2x^{22}y^4,\\
    g(x) &=1 - x^4 y- x^5 y^2  - x^6 y^3 - 3 x^7 y - 2 x^8 y^2 - x^9 y^3 - x^{18} y^3 - x^{21} y^3 + 2 x^{11} y^2 \\ &\quad+ 2 x^{12} y^3 + 3 x^{14} y^2 + 2 x^{15} y^3.
\end{align*}
    Propositions \ref{prop: ItoA} and \ref{prop: scaleIndex} yield the desired results.
\end{proof}

\section{Index of other acyclic meanders}\label{sec: moreacyclic}
In this section, we enumerate meanders on $J$ with a given index for other acyclic sets $J$ that are not covered by Lemma \ref{lem: acyclicRecursive}.

\begin{theorem}
    If $J \subseteq \{1,2,3\}$ or $J \subseteq \{1,3,4\}$, then $J$ is acyclic.
\end{theorem}

\begin{proof}
    Each node in our meander lies in at most one arc from the top of the diagram (upper arc) and at most one arc from the bottom of the diagram (lower arc). Every node of a cycle lies on exactly one upper and one lower arc, and both of those arcs belong to the cycle.

    Now let $\frac{\lambda}{\mu} \in \mathcal{I}_n^J$ with $n > \max(J)$, and suppose it contains a cycle $C$. Let $L$ be the minimal node of $C$, so that both arcs of $C$ at $L$ exit to the right. The parts of $\lambda$ and $\mu$ containing $L$ cannot occupy the same set of nodes, otherwise the meander is decomposable since $n > \max(J)$.

    First, consider the case where $J \subseteq \{1,2,3\}$. $L$ is the first node of both of the parts that contain it, and by the previous paragraph these parts must be of different sizes. Up to symmetry, the part of $\lambda$ containing $L$ has size $3$ and the part of $\mu$ has size $2$. The lower arc at $L$ then joins $L$ and $L+1$, so $L+1$ lies on $C$. However, $L+1$ is the middle node of a part of $\lambda$ of size $3$, so it lies on no upper arc, contradicting that $C$ is a cycle. This can be visualized as the meander containing
    \[
    \begin{tikzpicture}[
    dot/.style={circle, fill, inner sep=1.5pt},
    arc/.style={line width=0.6mm}, x=1.6cm,
    stuck/.style={circle, draw, thick, minimum size=11pt, inner sep=0pt}
]
    \draw[arc, blue] (0,0) to[bend left=45] (2,0);
    \draw[arc, red] (0,0) to[bend right=45] (1,0);
    \foreach \x in {0,1,2} { \node[dot] at (\x,0) {}; }
    \node[stuck] at (1,0) {};
    \node at (0,-0.75) {$L$};
    \node at (1,-0.75) {$L+1$};
    \node at (2,-0.75) {$L+2$};
\end{tikzpicture}
    \]
    with the circled vertex indicating it could not possibly be part of a cycle. Next, consider the case that $J \subseteq \{1,3,4\}$. As before, $L$ must be the first node of any part of size $3$ containing it. In a part of size $4$, $L$ is either the first node, lying on the outer arc, or the second node, lying on the inner arc, since the arcs at the third and fourth nodes exit to the left. If both parts containing $L$ had size $3$, they would both begin at $L$ and occupy the same nodes. So, up to symmetry, either both parts containing $L$ have size $4$, or the part of $\lambda$ has size $3$ and the part of $\mu$ has size $4$.

    In the first case, to be indecomposable, $L$ must be contained in an outer arc of one of the parts of size four and an inner arc of the other. Up to symmetry, assume L is contained in an inner arc of the top part and an outer arc of the bottom part. This means that our meander contains the following sub-diagram,
    \[
    \begin{tikzpicture}[
    dot/.style={circle, fill, inner sep=1.5pt},
    arc/.style={line width=0.6mm}, x=1.6cm,
    stuck/.style={circle, draw, thick, minimum size=11pt, inner sep=0pt}
]
    \draw[arc, blue] (0,0) to[bend left=45] (3,0);
    \draw[arc, blue] (1,0) to[bend left=45] (2,0);
    \draw[arc, red] (1,0) to[bend right=45] (4,0);
    \draw[arc, red] (2,0) to[bend right=45] (3,0);
    \foreach \x in {0,1,2,3,4} { \node[dot] at (\x,0) {}; }
    \node[stuck] at (0,0) {};
    \node at (0,-1.35) {$L-1$};
    \node at (1,-1.35) {$L$};
    \node at (2,-1.35) {$L+1$};
    \node at (3,-1.35) {$L+2$};
    \node at (4,-1.35) {$L+3$};
\end{tikzpicture}
    \]
    contradicting that $L$ is the minimal entry in the cycle.
    
    In the second case, there are now two further cases to consider: the arc containing $L$ in $\mu$ is the outer arc of the part of size $4$ connecting $L$ and $L+3$, or the inner arc connecting $L$ and $L+1$. In the first case, since the top arc containing $L$ must be the outer arc of the part of size $3$, this implies our meander contains
    \[
    \begin{tikzpicture}[
    dot/.style={circle, fill, inner sep=1.5pt},
    arc/.style={line width=0.6mm}, x=1.6cm,
    stuck/.style={circle, draw, thick, minimum size=11pt, inner sep=0pt}
]
    \draw[arc, blue] (0,0) to[bend left=45] (2,0);
    \draw[arc, red] (0,0) to[bend right=45] (3,0);
    \draw[arc, red] (1,0) to[bend right=45] (2,0);
    \foreach \x in {0,1,2,3} { \node[dot] at (\x,0) {}; }
    \node[stuck] at (1,0) {};
    \node at (0,-1.35) {$L$};
    \node at (1,-1.35) {$L+1$};
    \node at (2,-1.35) {$L+2$};
    \node at (3,-1.35) {$L+3$};
\end{tikzpicture}
    \]
    as a subdiagram. Again, this clearly contradicts our cyclic assumption.

    In the second case, our meander contains,
    \[
    \begin{tikzpicture}[
    dot/.style={circle, fill, inner sep=1.5pt},
    arc/.style={line width=0.6mm}, x=1.6cm,
    stuck/.style={circle, draw, thick, minimum size=11pt, inner sep=0pt}
]
    \draw[arc, blue] (1,0) to[bend left=45] (3,0);
    \draw[arc, red] (0,0) to[bend right=45] (3,0);
    \draw[arc, red] (1,0) to[bend right=45] (2,0);
    \foreach \x in {0,1,2,3} { \node[dot] at (\x,0) {}; }
    \node[stuck] at (2,0) {};
    \node at (0,-1.35) {$L-1$};
    \node at (1,-1.35) {$L$};
    \node at (2,-1.35) {$L+1$};
    \node at (3,-1.35) {$L+2$};
\end{tikzpicture}
    \]
    as a subdiagram, again contradicting the assumption that $C$ is cyclic and that $L$ is the minimal entry in the possible cycle. This exhausts all possible cases and our result follows.
\end{proof}

We note that this cannot be improved to $J \subseteq \{1,2,3,4\}$ since $\{2,4\}$ is not acyclic; $\frac{2|4}{4|2}$ contains a cycle and is indecomposable.

\begin{theorem} \label{thm: 23}
    We have
    \[
    A_{\{2,3\}}(x,y) =\frac{1 - x - x^{3} y}{(1-x^2y^2-x^3y^3)(1-x-x^3y)-2x^5y}
    \]
    and for any $c \geq 1$, $A_{\{2c,3c\}}(x,y) = A_{\{2,3\}}(x^c,y^c)$.
\end{theorem}

\begin{proof}
    For the duration of this proof, let $i_{n,k} \coloneqq i_{n,k}^{\{2,3\}}$. Given an indecomposable meander $\frac{a_1 | a_2 | \ldots | a_m}{b_1 | b_2 | \ldots | b_t} \in \mathcal{I}_n^{\{2,3\}}$ with index $k$, there are two options for how it could start: $\frac{3|\alpha}{2|\beta}$ or $\frac{2|\beta}{3|\alpha}$, for some $\alpha\vdash n-3$ and $\beta \vdash n-2$. For simplicity, let $r_{n,k}$ denote the number of irreducible seaweeds beginning with $\frac{3}{2}$ so that $i_{n,k} = 2r_{n,k}$ except $i_{2,2} = i_{3,3} = 1$ while $r_{2,2} = r_{3,3} = 0$. 

    There are two cases to consider, first if $b_2 = 2$ then we have
    \[
    \frac{3 | a_2 |a_3 | \ldots | a_m}{2 |2 | b_3 | \ldots | b_t}
    \]
    Notice that $\frac{2|2}{3|1}$ and $\frac{3}{2|1}$ both have index $1$ and have the same $1$-connectivity word. Theorem \ref{thm: shiftStart} then implies that the number of meanders of length $n$ with index $k$ beginning with $\frac{2|2}{3}$ is equal to the number of meanders of length $n-1$ with index $k$ beginning with $\frac{3}{2}$, which is precisely $r_{n-1,k}$.

    If instead $b_2 = 3$, we have
    \[
    \frac{3 | a_2 |a_3 | \ldots | a_m}{2 |3 | b_3 | \ldots | b_t}
    \]    which forces $a_2 = 3$ for the meander to remain indecomposable. But now $\frac{3|3}{2|3|1}$ and $\frac{3}{2|1}$ have the same $1$-connectivity word. The index of $\frac{3|3}{2|3|1}$ is $2$, while the index of $\frac{3}{2|1}$ is $1$. Theorem \ref{thm: shiftStart} implies that the number of meanders of length $n$ with index $k$ beginning with $\frac{3|3}{2|3}$ is equal to the number of meanders of length $n-3$ with index $k-1$ beginning with $\frac{3}{2}$, which is precisely $r_{n-3,k-1}$. Using the relationship that $i_{n,k} = 2r_{n,k}$, these observations imply that for $n > 6$,
    \[
    i_{n,k} = i_{n-1,k} + i_{n-3,k-1}.
    \]
    Accounting for initial conditions, this gives us 
    \[
    I_{\{2,3\}}(x,y) = x^2y^2 + x^3y^3 + \frac{2x^5y}{1-x-x^3y}.
    \]
    Propositions \ref{prop: ItoA} and \ref{prop: scaleIndex} yield the desired results.
\end{proof}

\begin{theorem} \label{thm: 34}
    We have 
    \[
    A_{\{3,4\}}(x,y) = \frac{1-x^2-x^3y-x^4+x^6}{(1- x^3y^3 - x^4y^4)(1-x^2-x^3y-x^4+x^6) - 2x^7y + 2x^9 y}
    \]
    and for any $c \geq 1$, $A_{\{3c,4c\}}(x,y) = A_{\{3,4\}}(x^c,y^c)$.
\end{theorem}

\begin{proof}
    For the duration of this proof, let $i_{n,k} \coloneqq i_{n,k}^{\{3,4\}}$. Given an indecomposable meander $\frac{a_1 | a_2 | \ldots | a_m}{b_1 | b_2 | \ldots | b_t} \in \mathcal{I}_n^{\{3,4\}}$ with index $k$, there are two options for how it could start: $\frac{4|\alpha}{3|\beta}$ or $\frac{3|\beta}{4|\alpha}$ for some $\alpha\vdash n-4$ and $\beta \vdash n-3$. For simplicity, let $r_{n,k}$ denote the irreducible seaweeds beginning with $\frac{4}{3}$, so that $i_{n,k} = 2 r_{n,k}$ except $i_{4,4} = i_{3,3} = 1$, while $r_{4,4} = r_{3,3} = 0$. 

    In our meander, if $b_2 = 4$, indecomposability forces $a_2 = 4$, so we have
    \[
    \frac{4|4|a_3|\ldots|a_m}{3|4|b_3|\ldots|b_t}.
    \]
    Notice both $\frac{4|4}{3|4|1}$ and $\frac{4}{3|1}$ have  $1$-connectivity word $1$ and index $1$. Theorem \ref{thm: shiftStart} implies that the number of meanders of length $n$ with index $k$ beginning with $\frac{4|4}{3|4}$ is equal to the number of meanders of length $n-4$ with index $k$ beginning with $\frac{4}{3}$. These are precisely enumerated by $r_{n-4,k}$.

    If $b_2 = 3$ and $a_2 = 3$, we have
    \[
    \frac{4|3|a_3|\ldots|a_m}{3|3|b_3|\ldots|b_t}.
    \]
    Using Theorem \ref{thm: shiftStart}, $\frac{4|3}{3|3|1}$ and $\frac{4}{3|1}$ have the same $1$-connectivity word. The index of $\frac{4|3}{3|3|1}$ is $2$ and the index of $\frac{4}{3|1}$ is $1$. So the number of meanders of length $n$ and index $k$ beginning with $\frac{4|3}{3|3}$ is equal to the number of meanders of length $n-3$ and index $k-1$ beginning with $\frac{4}{3}$. This is precisely $r_{n-3,k-1}$.

    Finally if $b_2 = 3$ and $a_2 = 4$, we have
    \[
    \frac{4|4|a_3|\ldots|a_m}{3|3|b_3|\ldots|b_t}.
    \]
    Notice $\frac{4|4}{3|3|1|1}$ and $\frac{3|3}{4|1|1}$ both have $2$-connectivity word $21$, and index $2$. Theorem \ref{thm: shiftStart} then implies that the number of indecomposable meanders beginning with $\frac{4|4}{3|3}$ of length $n$ and index $k$ is the same as the number beginning with $\frac{3|3}{4}$ of length $n-2$ and index $k$. These are enumerated by $r_{n-2,k} - r_{n-6,k}$. Indeed, we begin with all indecomposable meanders of length $n-2$ and index $k$ and remove those beginning with $\frac{4}{3|4}$. But this is the first case we considered, so these are enumerated by $r_{n-6,k}$. In all we find, for $n > 10$
    \[
    i_{n,k} = i_{n-2,k} + i_{n-3,k-1} + i_{n-4,k} - i_{n-6,k}.
    \]
    Accounting for initial conditions, this shows
    \[
    I_{\{3,4\}}(x,y) = x^3y^3 + x^4y^4 + \frac{2x^7y - 2x^9y}{1-x^2-x^3y-x^4+x^6}
    \]
    Propositions \ref{prop: ItoA} and \ref{prop: scaleIndex} yield the desired results.
\end{proof}

Notice $i^{\{3,4\}}_{n,k}$ follows the same recurrence as $i^{\{1,4\}}_{n,k}$ in Theorem \ref{thm: 14Indec}, but with initial conditions on $n$ shifted by 6. 

\begin{theorem} \label{thm: 123}
    We have $A_{\{1,2,3\}}(x,y)$ is
    \[
    \frac{1-x-x^2y-x^3y}{(1-x-x^2y-x^3y)(1-xy - x^2y^2 - x^3y^3) - (2x^2y + 4x^3y + 2x^3y^2 + 2x^4y^2 + 2x^5y)}
    \]
    and for any $c \geq 1$, $A_{\{c,2c,3c\}}(x,y) = A_{\{1,2,3\}}(x^c,y^c)$.
\end{theorem}

\begin{proof}
    For the duration of this proof, let $i_{n,k} \coloneqq i_{n,k}^{\{1,2,3\}}$. Given an indecomposable meander $\frac{a_1 | a_2 | \ldots | a_m}{b_1 | b_2 | \ldots | b_t} \in \mathcal{I}_n^{\{1,2,3\}}$ with index $k$, there are three options for how it could start, up to swapping the upper and lower composition: $\frac{2|\beta}{1|\alpha}$, $\frac{3|\gamma}{1|\alpha}$ or $\frac{3|\gamma}{2|\beta}$ for some $\alpha\vdash n-1$, $\beta \vdash n-2$ and $\gamma \vdash n-3$. Let $r_{n,k}$, $s_{n,k}$ and $t_{n,k}$ denote the number of these indecomposable meanders beginning with $\frac{2}{1}$, $\frac{3}{1}$ and $\frac{3}{2}$ respectively. We will consider each case separately.

    First, if $a_1 = 2$ and $b_1 = 1$, notice both $\frac{2}{1|1}$ and $\frac{1}{1}$ have the same $1$-connectivity word and index. Theorem \ref{thm: shiftStart} implies that $r_{n,k}$ equals the number of indecomposable meanders beginning with $\frac{1}{}$ of length $n-1$ and index $k$, so $r_{n,k} = r_{n-1,k} + s_{n-1,k}$.

    Next, if $a_1 = 3$ and $b_1  = 1$, to be indecomposable we must have $b_2 = 1$ or $b_2 = 3$. If $b_2 = 1$, both $\frac{3}{1|1|1}$ and $\frac{2}{1|1}$ have the same $1$-connectivity word. $\frac{3}{1|1|1}$ has index $2$, while $\frac{2}{1|1}$ has index $1$. Theorem \ref{thm: shiftStart} implies that the number of indecomposable meanders beginning with $\frac{3}{1|1}$ of length $n$ and index $k$ is equal to the number of indecomposable meanders beginning with $\frac{2}{1}$ of length $n-1$ and index $k-1$. This is exactly $r_{n-1,k-1}$.

    If instead $b_2 = 3$, both $\frac{1|3}{3|1}$ and $\frac{3}{2|1}$ have the same $1$-connectivity word. $\frac{1|3}{3|1}$ has index $2$, while $\frac{3}{2|1}$ has index $1$. Theorem \ref{thm: shiftStart} implies that the number of indecomposable meanders beginning with $\frac{3}{1|3}$ of length $n$ and index $k$ equals $t_{n-1,k-1}$

    Finally, if $a_1 = 3$ and $b_1 = 2$, notice both $\frac{3}{2|1}$ and $\frac{2}{1|1}$ have the same $1$-connectivity word and index. Theorem \ref{thm: shiftStart} implies $t_{n,k} = r_{n-1,k}$.

    In all we find
    \begin{align*}
    \tfrac{1}{2} i_{n,k} &= r_{n-1,k} + s_{n-1,k} +r_{n-1,k-1} + t_{n-1,k-1} + r_{n-1,k}
    \end{align*}
    However, $r_{n-1,k} = r_{n-2,k} + s_{n-2,k} = t_{n-1,k} + r_{n-3,k-1} + t_{n-3,k-1}$, and $t_{n-1,k-1} = r_{n-2,k-1} = r_{n-3,k-1} + s_{n-3,k-1}$, and finally $r_{n-1,k-1} = r_{n-2,k-1} + s_{n-2,k-1}$. Making these substitutions, the right hand side of our equality becomes
    \begin{align*}
    &r_{n-1,k} + s_{n-1,k} + (r_{n-2,k-1} + s_{n-2,k-1}) + (r_{n-3,k-1} + s_{n-3,k-1}) + (t_{n-1,k} + r_{n-3,k-1} + t_{n-3,k-1})\\
    = \, & (r_{n-1,k} + s_{n-1,k} + t_{n-1,k}) + (r_{n-2,k-1} + s_{n-2,k-1} + t_{n-2,k-1}) + (r_{n-3,k-1} + s_{n-3,k-1} + t_{n-3,k-1})\\
    = \, & \tfrac{1}{2} (i_{n-1,k} + i_{n-2,k-1} + i_{n-3,k-1}).
    \end{align*}
    This implies for $n > 6$, $i_{n,k} = i_{n-1,k} + i_{n-2,k-1} + i_{n-3,k-1}$. Accounting for initial conditions, we then find
    \[
    I_{\{1,2,3\}}(x,y) = xy + x^2y^2 + x^3y^3 + \frac{2x^2y + 4x^3y + 2x^3y^2 + 2x^4y^2 + 2x^5y}{1-x-x^2y-x^3y}.
    \]
    Propositions \ref{prop: ItoA} and \ref{prop: scaleIndex} yield the desired results.
\end{proof}

\begin{theorem} \label{thm: 134}
    We have that $A_{\{1,3,4\}}(x,y) = \frac{f(x)}{g(x)}$ where
    \begin{align*}
    f(x) &= 1-2x+2x^2-2x^3+x^4-x^2y\\
    g(x) &= 1-x(2+y) +x^2(2+y) - x^3(2+2y+y^2+y^3) +x^4(1-2y+2y^3-y^4) \\
    &\quad + x^5(5y-2y^3+3y^4) + x^6(-4y+2y^3-2y^4+y^5) + x^7(2y-y^3+2y^4) - x^8y^4
    \end{align*}
    and for any $c \geq 1$, $A_{\{c,3c,4c\}}(x,y) = A_{\{1,3,4\}}(x^c,y^c)$.
\end{theorem}

\begin{proof}
    For the duration of this proof, let $i_{n,k} \coloneqq i_{n,k}^{\{1,3,4\}}$. Given an indecomposable meander $\frac{a_1 | a_2 | \ldots | a_m}{b_1 | b_2 | \ldots | b_t} \in \mathcal{I}_n^{\{1,3,4\}}$ with index $k$, there are three options for how it could start, up to reversing the top and bottom: $\frac{3}{1}$, $\frac{4}{1}$ or $\frac{4}{3}$. Let $r_{n,k}$, $s_{n,k}$ and $t_{n,k}$ denote the number of these indecomposable meanders beginning with $\frac{3}{1}$, $\frac{4}{1}$ and $\frac{4}{3}$ respectively. We will consider each case separately.

    First if $a_1 = 3$ and $b_1 = 1$ there are three cases to consider, either $b_2 = 3$, $b_2 = 1$ or $b_2 = 4$. If $b_2 = 3$, then notice $\frac{1|3}{3|1}$ and $\frac{4}{3|1}$ have the same $1$-connectivity word. However $\frac{1|3}{3|1}$ has index $2$ while $\frac{4}{3|1}$ has index $1$. Theorem \ref{thm: shiftStart} implies that the number of indecomposable meanders beginning with $\frac{3}{1|3}$ of length $n$ and index $k$ is equal to the number of indecomposable meanders beginning with $\frac{4}{3}$ of length $n$ and index $k-1$. This is exactly $t_{n,k-1}$.

    If instead $b_2 = 1$, then $b_3$ can be $3$ or $4$. If $b_3 = 3$, $\frac{1|1|3}{3|1|1}$ and $\frac{3}{1|1|1}$ have the same $2$-connectivity word. $\frac{1|1|3}{3|1|1}$ has index $3$ while $\frac{3}{1|1|1}$ has index $2$. Theorem \ref{thm: shiftStart} implies that the number of indecomposable meanders beginning with $\frac{3}{1|1|3}$ of length $n$ and index $k$ is equal to the number of indecomposable meanders beginning with $\frac{3}{1}$ of length $n-2$ and index $k-1$. This is exactly $r_{n-2,k-1}$. If $b_3 = 4$, similar reasoning shows that these are enumerated by $s_{n-2,k-1}$. 
    
    Finally, if $b_2 = 4$, $\frac{4|1}{3|1|1}$ has the same $2$-connectivity word and index as $\frac{3}{1|1|1}$. Theorem \ref{thm: shiftStart} shows this case is enumerated by $r_{n-2,k}$.
    
    In all this shows
    \[
    r_{n,k} = t_{n,k-1} + r_{n-2,k} + r_{n-2,k-1} + s_{n-2,k-1}.
    \]
    Next, we consider if $a_1 = 4$ and $b_1 = 1$. In this case we must have $b_2 = 1$ or $b_2 = 4$. If $b_2 = 1$, then $\frac{4}{1|1|1|1}$ and $\frac{3}{1|1|1}$ have the same $2$-connectivity word and index. Theorem \ref{thm: shiftStart} implies that the number of indecomposable meanders beginning with $\frac{4}{1|1}$ of length $n$ and index $k$ is equal to the number of indecomposable meanders beginning with $\frac{3}{1}$ of length $n-1$ and index $k$. This is exactly $r_{n-1,k}$. If $b_2 = 4$, similar reasoning shows that these meanders are enumerated by $t_{n-1,k}$. In all we find, $s_{n,k} = r_{n-1,k} + t_{n-1,k}$.

    Finally if $a_1 = 4$ and $b_1 = 3$, we either have $b_2 = 3$ or $b_2 = 4$. In the first case, Theorem \ref{thm: shiftStart} implies these meanders are enumerated by $r_{n-3,k}$. Indeed, $\frac{3|3}{4|1|1}$ and $\frac{3}{1|1|1}$ have the same $2$-connectivity word and index. An identical argument shows if $b_2 = 4$ there are $s_{n-3,k}$ such meanders. We conclude, $t_{n,k} = r_{n-3,k} + s_{n-3,k}$.

    To find the generating function for $A_{\{1,3,4\}}(x,y)$, we use the auxiliary generating functions $R(x,y) = \sum_{n,k \geq 0} r_{n,k} x^n y^k$, $S(x,y) = \sum_{n,k \geq 0} s_{n,k} x^n y^k$ and $T(x,y) = \sum_{n,k \geq 0} t_{n,k} x^n y^k$. Accounting for initial conditions, our recursions imply:
    \begin{align*}
        R(x,y) &= (x^2 + x^2y + x^3y)R(x,y) + (x^2y + x^3y)S(x,y) + x^3y^2+x^4y^2,\\
        S(x,y) &= (x+x^4)R(x,y) + x^4S(x,y) + x^4y + x^5y,\\
        T(x,y) &= x^3(R(x,y) + S(x,y)) + x^4y
.    \end{align*}
    Solving for $S(x,y)$ in terms of $R(x,y)$ we find
    \[
    S(x,y) = \frac{(x+x^4)R(x,y) + x^4y + x^5y}{1-x^4}.
    \]
    Substituting this into the functional equation for $R(x,y)$ and solving for $R(x,y)$, we find
    \[
    R(x,y) = \frac{y^2(x^3 + x^4 + x^6 + x^7)}{1-x^2-x^4+x^6-x^2y - 2x^3y - x^4y}.
    \]
    We can substitute this back into the functional equation for $S(x,y)$ and solve,
    \[
    S(x,y) = \frac{x^4y + x^4y^2 - x^5y-x^5y^2- x^8y - x^8y^2 + x^9y + x^9y^2}{1-2x+2x^2-2x^3+2x^5-2x^6+2x^7-x^8-x^2y+x^6y}.
    \]
    Finally, substituting both of these into the functional equation for $T(x,y)$ we find
    \[
    T(x,y) = \frac{x^4y - 2x^5y + 2x^6y - x^7y}{1-2x + 2x^2 - x^2y-2x^3+x^4}.
    \]
    \noindent Since $I_{\{1,3,4\}}(x,y) = 2(R(x,y) + S(x,y) + T(x,y)) + xy+x^3y^3 + x^4y^4$, we find
    \begin{align*}
    I_{\{1,3,4\}}(x,y) = xy + x^3y^3 + x^4y^4 + \frac{2x^3y^2 + 4x^4y - 6x^5y + 4x^6y - 2x^7y}{1-2x+x^2(2-y)-2x^3+x^4}.
    \end{align*}
    We note that this shows for $n > 8$, $i_{n,k}$ satisfies the recursion:
    \[
        i_{n,k} = 2i_{n-1,k} - 2i_{n-2,k} + i_{n-2,k-1} + 2i_{n-3,k} - i_{n-4,k}.
    \]
    Propositions \ref{prop: ItoA} and \ref{prop: scaleIndex} yield the desired results.
\end{proof}

\section{Some non-acyclic cases}\label{sec: cyclic}

In this section we consider the case when $J=\{1,6\}$ and $J=\{1,2,n\}.$ In the case when $J=\{1,6\}$ we cannot track the index since the set is not acyclic, but we are able to apply the recursion in Lemma \ref{lem: acyclicRecursive} to find the number of indecomposable meanders of length $n$ whose parts are restricted to $\{1,6\}$. In the case $J=\{1,2,n\},$ we are able to count meanders by their index despite $J$ being non-acyclic.

\begin{theorem} \label{thm: 16Indec}
We have
\[
I_{\{1,6\}}(x) = x + x^6 + \frac{2x^{6} + 2x^{7} + 2x^{8} - 2x^{10} - 4x^{11} - 2x^{13} + 2x^{16}}{1 - x^{3} - x^{4} - x^{5} - 2x^{6} - x^{7} + 2x^{9} + 2x^{10} + x^{12} - x^{15}}
\]
and for any $c \geq 1$, $I_{\{c,6c\}}(x) = I_{\{1,6\}}(x^c)$.
\end{theorem}

\begin{proof}
    For the remainder of this proof, let $i_{n}$ denote $i^{\{1,6\}}_{n}$. By iteratively applying the recursions in Lemma \ref{lem: acyclicRecursive} we obtain $c_{n}^1 = i_{n-6}$, $c_n^2=i_{n-6} - i_{n-12} + i_{n-7},$ $c_n^3 = i_{n-3} - 2i_{n-9} + i_{n-15} - i_{n-10}$, $c_n^4 = i_{n-4} -i_{n-10}$, and $c_n^5 = i_{n-5}.$
    In all, using the first recursion we find for $n > 21$,
    \begin{align*}
       i_{n} &= i_{n-3} + i_{n-4} + i_{n-5}  + 2 i_{n-6} + i_{n-7} - 2i_{n-9} - 2 i_{n-10} - i_{n-12} + i_{n-15}.
   \end{align*}
   Accounting for initial conditions, this proves the desired result.
\end{proof}

We note that $A_{\{1,6\}}(x)$ is just the generating function for pairs of compositions of $n$ composed solely of $1$'s and $6$'s. The more interesting result is actually $I_{\{1,6\}}(x)$, which imposes the additional indecomposability condition. For completeness, Proposition \ref{prop: ItoA} gives us
\[
A_{\{1,6\}}(x) =
\frac{1 - x^{3} - x^{4} - x^{5} - 2x^{6} - x^{7} + 2x^{9} + 2x^{10} + x^{12} - x^{15}}
{\left(1 - x - 2x^{3} + x^{6}\right)\left(1 + x^{3} + x^{4} + x^{5} - 2x^{6} - x^{7} - 2x^{9} - 2x^{10} + x^{12} + x^{15}\right)}.
\]
which is indeed the generating function for the squares of \cite[A005708]{OEIS}.

The proof of the next theorem requires counting meanders of the form $\frac{n}{b_1| b_2| \ldots |b_t}$ with $b_i \in \{1,2\}$. Ignoring the index, these are counted by the Fibonacci numbers.

\begin{theorem}\label{thm:12n}
    We have \[A_{\{1,2,n\}}(x,y) =A_{\{1,2\}}(x,y) + \frac{3x^{3}y^{3} + 2x^{3}y^{2} - 2x^{2}y - 2}{1 - xy} +\frac{2(1+xy-x^2+x^2y^2+x^3y-x^4y^2)}{1-x^2-x^2y-x^4y-x^4y^2+x^6y^2}\]
    where $A_{\{1,2\}}(x,y)$ is defined in Theorem~\ref{thm: 12Indec}.
\end{theorem}

\begin{proof}
    Let $r_{n,k}$ denote the number of meanders of the form $\frac{n}{b_1| b_2| \ldots |b_t}$ where $b_i \in \{1,2\}$, let $R(x,y) = \sum_{n,k \geq 0} r_{n,k} x^n y^k$. 
    First, notice that  
    \[A_{\{1,2,n\}}(x,y) = 2(R(x,y)-1-xy-x^2y^2 - x^2y) + A_{\{1,2\}}(x,y) + \frac{(xy)^3}{1-xy}\] because the possible meanders we obtain fall into four disjoint cases:
    \begin{itemize}
        \item $\frac{\lambda}{\mu} = \frac{n}{\alpha}$ with $n\geq 3$ and  $\alpha\vdash n$ with all parts in $\{1,2\}$,
        \item $\frac{\lambda}{\mu} = \frac{\alpha}{n}$ with $n\geq 3$ and $\alpha\vdash n$ with all parts in $\{1,2\}$,
        \item $\frac{\lambda}{\mu} = \frac{\alpha}{\beta}$ with $\alpha, \beta \vdash n$ with all parts in $\{1,2\}$,
        \item $\frac{\lambda}{\mu} = \frac{n}{n}$ with $n\geq 3.$
    \end{itemize}
    Clearly the last two cases are enumerated by $A_{\{1,2\}}(x,y) + \frac{(xy)^3}{1-xy}$ and the first and second cases are equal in number, and are enumerated by exactly $R(x,y)-1-xy-x^2y^2 - x^2y$. Thus it suffices to find $R(x,y).$
    
    Given a meander of the form $\frac{n}{b_1| b_2| \ldots |b_t}$ where $b_i \in \{1,2\}$, let $v_1,\dots,v_n$ be the ordered nodes in our meander. We will consider the number of nodes contained in the path that contains the first and last node in the meander. Since $a_1 = n$, $v_1$ and $v_n$ are necessarily in the same path (or cycle). 
    
    One exceptional case we must consider is if $v_1,v_2,v_{n-1}$ and $v_n$ form a cycle in the meander, which occurs if the bottom composition of the meander looks like $2|b_2|\cdots|b_{t-1}|2$. In this case, if we remove this cycle, the remaining meanders are enumerated by $r_{n-4,k-2}$. This is the only way that the sub-meander containing $v_1$ and $v_n$ could be a cycle. Indeed, if $b_1 = 1$, $v_1$ is only connected to $v_n$ and if $b_t = 1$ then $v_n$ is only connected to $v_1$. 
    
    It remains to consider the case where $v_1$ and $v_n$ are contained in a path. We must consider the possible length of such a path, and whether the path begins with $v_1$ or $v_n$. Any such path will necessarily contain an even number of nodes, since it corresponds to the bottom composition being either $2^k|1|\beta|2^\ell|1$, $1|2^k|\beta|1|2^\ell$, $2^k|\beta|1|2^\ell|1$ or $1|2^k|1|\beta|2^\ell$, where $\beta \vdash  n-2(k+\ell+1)$ with all parts in $\{1,2\}$; we then remove the $2(k+\ell)+2$ nodes that are part of the path containing $v_1$ and $v_n$. 
    
    So the possible path lengths are $2j$ where $1 \leq j \leq \lfloor n/2 \rfloor$. If $j = k+\ell+1$, when we remove such a path from a meander of the form $2^k|1|\beta|2^\ell|1$ or $2^k|\beta|1|2^\ell|1$, the remaining meanders $\beta$ are enumerated by $r_{n-2j,k-1}$. These are precisely the meanders where the path of length $2j$ begins with $v_n$. Likewise, if the path begins with $v_1$ there are $r_{n-2j,k-1}$ such meanders. The only exception is when $j = 1$, in which case the path is length $2$, collapsing the different cases. When we remove this path of length $2$, the remaining meanders are enumerated by $r_{n-2,k-1}$. In all, we find for $n \geq 3$, $k \geq 2$
    \[
        r_{n,k} = r_{n-2,k-1} + r_{n-4,k-2} + 2\sum_{j=2}^{\lfloor n/2 \rfloor} r_{n-2j,k-1}
    \]
    and $r_{0,0} = r_{1,1} = r_{2,1} = r_{2,2} = 1$, and $r_{n,1} = 2$ for $n \geq 3$. This implies:
    \[
    R(x,y) = \frac{1+xy-x^2+x^2y^2+x^3y-x^4y^2}{1-x^2-x^2y-x^4y-x^4y^2+x^6y^2}.
    \]
    Since $A_{\{1,2,n\}}(x,y) = 2(R(x,y)-1-xy-x^2y^2 - x^2y) + \frac{(xy)^3}{1-xy} + A_{\{1,2\}}(x,y)$, we can obtain the formula in the statement of the theorem.
\end{proof}

\begin{example}
Consider the meander $\frac{9}{2|1|2|1|2|1}$ pictured below

    
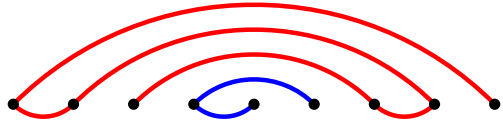
\begin{center}
\begin{tikzpicture}[ scale =.8,
    dot/.style={circle, fill, inner sep=1.5pt},
    arc/.style={line width=0.6mm}
]

    \draw[arc, red] (0,0) to[bend left=45] (8,0);
    \draw[arc, red] (1,0) to[bend left=45] (7,0);
    \draw[arc, red] (2,0) to[bend left=45] (6,0);
\draw[arc, blue] (3,0) to[bend left=45] (5,0);

    \draw[arc, red] (0,0) to[bend right=45] (1,0);
    \draw[arc, blue] (3,0) to[bend right=45] (4,0);
    \draw[arc, red] (6,0) to[bend right=45] (7,0);

    \node[dot] at (0,0) {};
    \node[dot] at (1,0) {};
    \node[dot] at (2,0) {};
    \node[dot] at (3,0) {};
    \node[dot] at (4,0) {};
    \node[dot] at (5,0) {};
    \node[dot] at (6,0) {};
    \node[dot] at (7,0) {};
    \node[dot] at (8,0) {};

    \end{tikzpicture}
    \end{center}

    In this case, the path containing $v_1$ and $v_9$ is in red. It begins with $v_9$ because the meander is of the form $2^1|1|\beta|2^1|1$, where here $\beta = 2|1$. Upon removing the red path, we are always left with a meander of the form $\frac{9-6}{\beta}$, in this case $\frac{3}{2|1}$.
\end{example}


\section{Relationship to permutations with restricted displacement}\label{sec: bijection}

In this section we establish a bijection between (half of the) indecomposable meanders with restricted part sizes and permutations with restricted displacements. Such permutations have been studied and enumerated in \cite{B10}. When the meanders are acyclic, we show how to recover the index of the meander from a statistic of the associated permutation. 
In this section, we use $\bar{\mathcal{I}}_n$ to denote those meanders $\frac\lambda\mu\in\mathcal{I}_n$ that have $\lambda_1>\mu_1.$ Notice that $|\bar{\mathcal{I}}_n|=\frac{1}{2}(|\mathcal{I}_n|-1)$ after accounting for the case when $\lambda=\mu=(n)$.

We start with some notation. Let $\S_n$ denote the set of permutations on $[n]:=\{1,2,\ldots,n\}$. We write permutations in one-line notation as $\pi = \pi_1\pi_2\ldots\pi_n$ where $\pi_i:=\pi(i)$. The \emph{displacement} of $\pi_i$ is $\pi_i-i$, and we set $d(\pi)=(\pi_1-1, \pi_2-2, \ldots, \pi_n-n)$ for $\pi=\pi_1\pi_2\ldots\pi_n\in\S_n$. We use $\Pi_{n}^J$ to denote the set of permutations $\pi$ of $[n]$ so that $d(\pi)$ consists only of elements from $J$. 

\begin{lemma}\label{lem:displacement}
    For $n\geq 1$, if $\pi \in \Pi_n^{J}$ with $J=\{-2,-1,j-2\}$, and $d(\pi) = (d_1,\ldots, d_n),$ then $d_1=j-2$. If $r$ is the smallest integer greater than 1 so that $d_r=j-2,$ then 
    \begin{itemize}
        \item $r\leq j,$
        \item $d_i=-1$ for $2\leq i \leq r-1$, and $d_\ell=-2$ for $r+1\leq \ell \leq j$, and 
        \item  $d'=(d_1',\ldots,d'_{n-r+1})$, obtained by taking $d_1'=j-2,$ $d_i'=-1$ for $2\leq i\leq j-r+1$, and $d_i'=d_{i+r-1}$ for $i>j-r+1,$ is a displacement list for some $\pi'\in\Pi_{n-r+1}^J,$
    \end{itemize}
\end{lemma}

\begin{proof}
    First notice that $\pi_1\geq 1,$ and so $d_1\geq 0.$ Since here, we require $d_i\in\{-2,-1,j-2\}$, we must have $d_1=j-2.$ 

    Now let $d_r$ be the next occurrence of $j-2$ after $d_1.$ That is, $d_r=j-2$ and $d_i\in\{-1,-2\}$ for $1< i<r$. To prove the first bullet, suppose for the sake of contradiction that $r\geq j+1$. Among $\{d_i+i : 1\leq i \leq j\}$, there would be two equal elements (since all elements are less than $j$), and thus the elements $d_1,\ldots,d_n$ could not have come from the displacement list of a permutation. 

    To prove the second bullet point, first suppose $r=2.$ In that case, we must have had $\pi_1\pi_2 = (j-1)j$. Since all displacements are in $\{-2,-1,j-2\}$, we must have that $d_\ell=-2$ for $3\leq \ell\leq j$ in order for the elements in $\{1,\ldots, j-2\}$  to appear in $\pi$. If $r>2,$ then the elements of $\pi$ after position $r$ must all be greater than or equal to $r-1$. Thus we must have the elements $\{\pi_2,\ldots, \pi_{r-1}\}=\{1,\ldots, r-2\}.$ Since the only options for $\pi_i-i$ are $\{-1,-2\},$ they must all in fact be $-1.$ Furthermore, after $\pi_r$, we must similarly have all the elements less than $j-1$ appear first and so $d_\ell=-2$ for $r+1\leq \ell\leq j.$

    Finally, let us prove the third bullet point. This states that we can delete $d_1,\ldots,d_{r-1}$, and shift elements appropriately to still have a valid displacement list. Indeed, by deleting $(j-1)12\ldots (r-2)$ from the permutation $\pi$, we subtract $r-1$ from all elements greater than $j-1$ and $r-1$ from the position of all such elements. The only elements affected differently are those $d_i$ with $2<i<j$. In this case, the elements that were previously $-2$ are now $-1$ since the position was changed by $r-1$ but the values of $\pi$ were changed by $r-2.$ 
\end{proof}

Let us see an example to demonstrate the last bullet point of Lemma~\ref{lem:displacement}.

\begin{example}
    Let $j=5$ and thus $J = \{-2,-1,3\}$.  One example of a permutation in $\Pi_{10}^J$ is $\pi = 412735(10)689$, which has $d(\pi) = (3,-1,-1,3,-2,-1,3,-2,-1,-1).$
    By the last bullet point of  Lemma~\ref{lem:displacement}, we have that $r=4$ and take $d'=(3,-1,-1,3,-2,-1,-1)$ where $d'=d(\pi')$ for $\pi' = 4127356 \in \Pi_7^J.$
\end{example}

\begin{theorem}\label{thm:bijection}
    Let $j\geq 2$. 
    Then for $n>j$ there exists a bijection between $\bar{\mathcal{I}}_n^{\{1,j\}}$ and $\Pi_{n-1}^{\{-2,-1,j-2\}}$. 
\end{theorem}

To prove Theorem \ref{thm:bijection}, we will establish the explicit bijection between $\Pi_{n-1}^{\{-2,-1,j-2\}}$ and $\bar{\mathcal{I}}_n^{\{1,j\}}$. To this end, we define two maps \[f: \Pi_{n-1}^{\{-2,-1,j-2\}}\longrightarrow \bar{\mathcal{I}}_n^{\{1,j\}}\quad \text{ and } \quad g:\bar{\mathcal{I}}_n^{\{1,j\}}\to \Pi_{n-1}^{\{-2,-1,j-2\}},\] and show they are inverses.

In order to define $f$, we will establish an algorithm that uses the displacement list, $d(\pi)$, of a permutation $\pi \in \Pi_{n-1}^{\{-2,-1,j-2\}}$ as a sequence of instructions to  build a pair of compositions, $\lambda$ and $\mu$, which in turn form an irreducible meander $\frac{\lambda}{\mu}\in\bar{\mathcal{I}}_n^{\{1,j\}}$. 

\begin{definition}\label{def:f} Here we define a map $f$ that takes in an element of $\Pi_{n-1}^{\{-2,-1,j-2\}}$ and returns a pair of compositions which we will later see in Lemma~\ref{lem:f works} correspond to an indecomposable meander in $\bar{\mathcal{I}}_n^{\{1,j\}}$.
Let $\pi \in \Pi_{n-1}^{\{-2,-1,j-2\}}$ with $d(\pi) = \{d_1, \ldots, d_{n-1}\}$. We describe $f(\pi)$ via an algorithm. 

Note that $d_1=j-2$ by necessity. Set $\lambda_1=j$ and $\mu_1=1$. We initialize a state tracker \texttt{last\_j} to be $\lambda$. (At each step of the algorithm, \texttt{last\_j}  will be either $\lambda$ or $\mu$.) We proceed by reading remaining elements of the displacement list $d_2,\ldots,d_{n-1}$ from left to right. Here the composition \emph{opposite} a composition $\nu\in\{\lambda,\mu\}$ means the other one, i.e., if $\mu$ if $\nu=\lambda$ and $\lambda$ if $\nu=\mu$. At step $i$:
\begin{itemize}
 \item If $d_i=j-2$, then append a part size of $j$ to the composition opposite the \texttt{last\_j} and update the \texttt{last\_j} to this composition,
        \item If $d_i=-1$, then append a part size of 1 to the composition opposite \texttt{last\_j},
        \item If $d_i=-2$, then do nothing. 
\end{itemize}
After we handle $d_{n-1}$, we add a 1 to the composition opposite the \texttt{last\_j}. 
We define $f$ by $f(\pi) = \frac{\lambda}{\mu}$.
\end{definition}

\begin{example}
    We demonstrate the algorithm for $f$ for a specific case with $n=9$ and $j=4$. Let $\pi=31264857 \in S_8$. We have $d(\pi) = (2, -1, -1, 2, -1, 2, -2, -1)$ and so $\pi \in\Pi_{8}^{\{-2,-1,2\}}$. Note $d_1=2$ and we start with $\lambda_1=4$, $\mu_1=1$, and $\texttt{last\_j}=\lambda$. Proceeding, we get
    \begin{multline*}
   \fbox{$\dfrac{4}{1}$}\xrightarrow[]{d_2=-1}\fbox{$\dfrac{4}{1|1}$}\xrightarrow[]{d_3=-1}\fbox{$\dfrac{4}{1|1|1}$}\xrightarrow[\texttt{last\_j}=\mu]{d_4=2}\fbox{$\dfrac{4}{1|1|1|4}$}\xrightarrow[]{d_5=-1}\fbox{$\dfrac{4|1}{1|1|1|4}$} \\ \xrightarrow[\texttt{last\_j}=\lambda]{d_6=2}\fbox{$\dfrac{4|1|4}{1|1|1|4}$} \xrightarrow[]{d_7=-2}\fbox{$\dfrac{4|1|4}{1|1|1|4}$}\xrightarrow[]{d_8=-1}\fbox{$\dfrac{4|1|4}{1|1|1|4|1}$}\xrightarrow[]{}\fbox{$\dfrac{4|1|4}{1|1|1|4|1|1}$}
    \end{multline*}
    where the last step is to add a 1 at the end of the composition opposite \texttt{last\_j}, which in this case is $\mu$. Note that $\dfrac{4|1|4}{1|1|1|4|1|1} \in \bar{\mathcal{I}}_9^{\{1,4\}}$.
\end{example}

\begin{lemma}\label{lem:f works}
For $f$ as defined in Definition~\ref{def:f}, if $\pi \in \Pi_{n-1}^{\{-2,-1,j-2\}}$, then $f(\pi) \in \bar{\mathcal{I}}_n^{\{1,j\}}$. 
\end{lemma}

\begin{proof}
To complete this proof, it suffices to prove that $\lambda$ and $\mu$ are both compositions of $n$ and that their sets of partial sums $PS(\lambda)$ and $PS(\mu)$ have trivial intersection (that is, that the intersection is $\{n\}$). 

To show that $\lambda$ and $\mu$ are both compositions of $n$, let us first show that $|\lambda|+|\mu|=2n$, and then that $|\lambda|=|\mu|.$ For the former statement, notice that 
\begin{itemize}
    \item each occurrence of $j-2$ in $d(\pi)$ contributes $j$ to the sum $|\lambda|+|\mu|$, except for the first occurrence which contributes an extra 1
    \item each occurrence of $-1$ in $d(\pi)$ contributes 1 to the sum $|\lambda|+|\mu|$,
    \item each occurrence of $-2$ contributes 0 to the sum $|\lambda|+|\mu|$,
    \item and there is an additional 1 added to the sum $|\lambda|+|\mu|$ at the end of the algorithm.
\end{itemize}
Taken altogether, this implies that if $d(\pi) = (d_1,\ldots, d_{n-1}),$ then  \[|\lambda|+|\mu| = 2+ \sum_{i=1}^{n-1} (d_i+2).\]
Since $\pi$ is a permutation, it follows that $\sum d_i=0$, and thus $|\lambda|+|\mu|=2+2(n-1) = 2n$.

Now, to prove that $|\lambda|=|\mu|$, we will proceed by strong induction.  For the base case, we have the smallest permutation with displacement list $J$, which happens when $n=j-1$, is $(j-1)12\ldots(j-2)$, and maps to $\frac{j}{1^j}.$ Clearly $|\lambda|=|\mu|$ in this case. Note that this is the only permutation $\pi$ with displacement list $J$ that has only one displacement equal to $j-2$.
Now suppose that for all $n<N,$ the algorithm produces compositions $\lambda'$ and $\mu'$ of the same size. Suppose we apply this algorithm  to a permutation $\pi\in\Pi_{N-1}^J$ and we obtain $\lambda$ and $\mu$. Notice that since $\pi\in\Pi_{N-1}^J$, its displacement list $d(\pi)$ must satisfy the conditions of Lemma~\ref{lem:displacement}. In particular, the last bullet says that there is a smaller permutation $\pi'\in\Pi_{n-r+1}$ with $d(\pi')$ obtained from $d(\pi)$ by deleting the first $r-1$ elements of $d(\pi)$, and changing all the required $-2$'s from the second bullet point to $-1$'s. In the image of $\pi'$ under $f$, this corresponds to taking $f(\pi) = \frac{\lambda}{\mu}$ and obtaining $f(\pi') = \frac{\mu'}{\lambda'}$ obtained by deleting all leading 1's of $\mu$ to get $\mu'$, and breaking $\lambda_1=j$ into $j$ 1's, and deleting an equal number of 1's. Since an equal amount has been deleted from $\lambda$ and $\mu$, and by induction $|\lambda'|=|\mu'|,$ we must  have had that $|\lambda|=|\mu|.$
Note that each time you perform this action, we reduce the number of $(j-2)$'s in the displacement list by 1 and we reduce the number of $j$'s in the meander $f(\pi)$ by 1. Thus by performing this action over and over, we will end up back in the base case. 

Finally, we need to show that the set of partial sums of $\lambda$ and $\mu$ have trivial intersection (in order to show that the meander $\frac{\lambda}{\mu}$ is indecomposable). 
For the sake of contradiction, let $m<n$ be the minimal element in $PS(\lambda)\cap PS(\mu).$ Let us say $m=\lambda_1+\cdots + \lambda_s = \mu_1+\cdots + \mu_t$. Since $m$ is minimal, we must have that $\{\lambda_s,\mu_t\} = \{1,j\}$ (i.e., they are not equal and one of them is $j$). 
Without loss of generality, suppose $\lambda_s=j$ and $\mu_t=1.$ Then in the step where $j$ is added, we switch \texttt{last\_j} to $\lambda$ and so $\mu_t$ is added after $\lambda_s,$ say at the $p$-th step. Let us consider what must be true about $d(\pi)$ up until the $p$th element. Note that by the argument above, we must have that the partial sum of $d_1+\cdots + d_{p-1}=0$ since $\lambda_1+\cdots +\lambda_s=\mu_1+\cdots + \mu_t$. But since $\mu_t=1$, we must have $d_p=-1$, and thus  $d_1+\cdots + d_{p}=-1$, which is impossible if $\pi$ is a permutation. Thus $f(\pi)$ results in an indecomposable meander. 
\end{proof}

Now let us define $g$, which we eventually prove is the inverse of $f$. 

\begin{definition}\label{def:g}
We define $g$ such that a pair of compositions $\lambda, \mu$ with $\frac{\lambda}{\mu}\in\bar{\mathcal{I}}_n^{\{1,j\}}$ are inputted and a list of integers $d_1,\ldots, d_{n-1}$, with $d_i \in\{-2,-1,j-2\}$ for each $i$ is outputted. We will later show this is exactly the displacement list $d(\pi)$ of a permutation $\pi\in\Pi_{n-1}^{\{-2,-1,j-2\}}$. Note that by assumption, $\lambda_1=j$ and $\mu_1=1$.

First, we construct the sets of partial sums of $\lambda$ and $\mu$. For each $m \in[n-1]$ with $m\not\in PS(\lambda)\cup PS(\mu)$, set $d_{m+1}=-2$. Additionally set $d_1 = j-2$, the state tracker \texttt{last\_j} to $\lambda$, and mark the elements $\lambda_1=j$ and $\mu_1=1$ as read. Then for each $2\leq i \leq n-1$ with $d_i$ unassigned (i.e.~with $d_i\neq -2$), we will make an assignment:
\begin{itemize}
\item if the next unread element of the composition opposite \texttt{last\_j} is 1, set the earliest unassigned displacement to $-1$,
\item if the next unread element of the composition opposite \texttt{last\_j} is $j$, set the earliest unassigned displacement to $j-2$, and change the state tracker to the opposite composition.
\end{itemize}
The process is complete when we reach the last 1. The displacement list we obtain will allow us to reconstruct a permutation $\pi$ (by adding $i$ to $d_i$), and we set $g(\frac{\lambda}{\mu}) = \pi$. 
\end{definition}

\begin{example}
    We demonstrate the inverse by reconstructing the displacement list from the previous example.
 Consider $\dfrac{4|1|4}{1|1|1|4|1|1} \in \bar{\mathcal{I}}_9^{\{1,4\}}$. Then $PS(\lambda)=\{4,5,9\}$ and $PS(\mu)=\{1,2,3,7,8,9\}$. The only element $m\in [n-1]$ not present in the union of the partial sums is 6, so we set $d_7=-2$. We then walk through the compositions according to the algorithm building a list $d$.
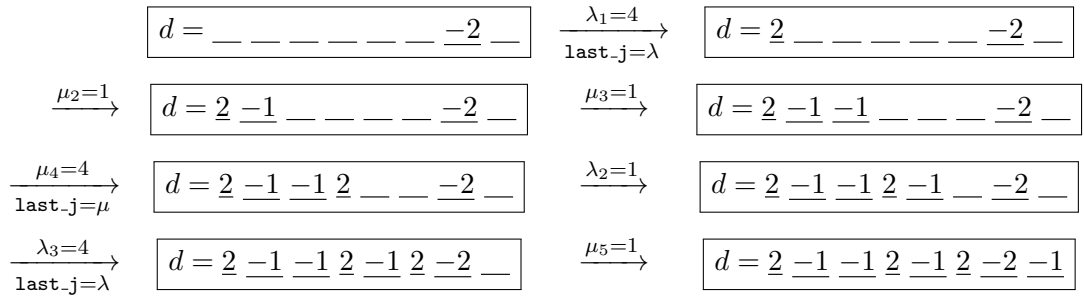
 \[
 \setlength{\fboxsep}{4pt}
\renewcommand{\arraystretch}{1.3}
\begin{array}{rccc}
 & \fbox{$d = \underline{\quad}\ \underline{\quad}\ \underline{\quad}\ \underline{\quad}\ 
\underline{\quad}\ \underline{\quad}\ \underline{-2}\ \underline{\quad}$}
& \xrightarrow[\texttt{last\_j}=\lambda]{\lambda_1=4} &
\fbox{$d = \underline{2}\ \underline{\quad}\ \underline{\quad}\ \underline{\quad}\ 
\underline{\quad}\ \underline{\quad}\ \underline{-2}\ \underline{\quad}$}
\\[1em]
 \xrightarrow[]{\mu_2=1} &
\fbox{$d = \underline{2}\ \underline{-1}\ \underline{\quad}\ \underline{\quad}\ 
\underline{\quad}\ \underline{\quad}\ \underline{-2}\ \underline{\quad}$} & 
\xrightarrow[]{\mu_3=1} & 
\fbox{$d = \underline{2}\ \underline{-1}\ \underline{-1}\ \underline{\quad}\ 
\underline{\quad}\ \underline{\quad}\ \underline{-2}\ \underline{\quad}$}
\\[1em]
 \xrightarrow[\texttt{last\_j}=\mu]{\mu_4=4} &
\fbox{$d = \underline{2}\ \underline{-1}\ \underline{-1}\ \underline{2}\ 
\underline{\quad}\ \underline{\quad}\ \underline{-2}\ \underline{\quad}$} & 
\xrightarrow[]{\lambda_2=1} & 
\fbox{$d = \underline{2}\ \underline{-1}\ \underline{-1}\ \underline{2}\ 
\underline{-1}\ \underline{\quad}\ \underline{-2}\ \underline{\quad}$}
\\[1em]
\xrightarrow[\texttt{last\_j}=\lambda]{\lambda_3=4} &
\fbox{$d = \underline{2}\ \underline{-1}\ \underline{-1}\ \underline{2}\ 
\underline{-1}\ \underline{2}\ \underline{-2}\ \underline{\quad}$}
& \xrightarrow[]{\mu_5=1} & 
\fbox{$d = \underline{2}\ \underline{-1}\ \underline{-1}\ \underline{2}\ 
\underline{-1}\ \underline{2}\ \underline{-2}\ \underline{-1}$}
\end{array}
\]

Note the final 1 in $\mu$ correlates to the final forced step in the algorithm for $f$ and therefore does not add anything to the displacement list we create here. 
Finally, we use the list $(d_1,\ldots, d_{n-1})$ constructed to obtain $\pi=31264857 \in \Pi^{\{-2,-1,2\}}_8$ by taking $\pi_i = d_i+i$.

\end{example}

\begin{lemma}\label{lem:g works}
For $g$ as defined in Definition~\ref{def:g}, if $\frac{\lambda}{\mu} \in \bar{\mathcal{I}}_n^{\{1,j\}}$, then $g(\frac{\lambda}{\mu}) \in \Pi_{n-1}^{\{-2,-1,j-2\}}$.
\end{lemma}

\begin{proof}
    By design the algorithm for $g$ will produce a displacement list of length $n-1$ consisting of elements from $\{-2,-1,j-2\}$. We need only show that by taking $\pi=\pi_1\ldots \pi_{n-1}$ with $\pi_i=d_i+i$, we obtain a permutation in $\S_{n-1}.$ To that end, we will show that $1\leq \pi_i\leq n-1$ for all $i$ and that $\pi_i\neq\pi_j$ if $i\neq j.$

    First, let us notice that $\pi_i\geq 1$ since $d_1$ is necessarily $j-2$ and $d_2$ cannot be $-2$ (since $1$ is always in $PS(\mu)$). Now, let us see that $\pi_i\leq n-1$. Indeed this can only be violated if $d_i=j-2$ for some $i>n-1-(j-2)$. Consider the last $j-2$ in $d_i$. This must correspond to either the last element of $\mu$ or of $\lambda$ (since they cannot both end in the same element as that would violate indecomposability). Without loss of generality, suppose $\mu$ ends with $j$ and $\lambda$ ends with $j1^\ell$ for some $\ell<j$. Notice that by the algorithm, since elements $r$ satisfying $n-j<r<n-\ell$ are not in $PS(\lambda)\cup PS(\mu)$, then $d_{r+1}=-2$ for those elements. Furthermore, the elements after that sequence of $-2$'s must be $-1$ to account for the $\ell$ 1's (except the last one). Thus, the latest $j-2$ must occur before position $n-1-(j-2).$

    Next let us see that $\pi_r\neq \pi_s$ if $r\neq s.$ For the sake of contradiction, suppose there is some $r\neq s$ for which this does happen, i.e. for which $d_r+r=d_s+s,$ and without loss of generality let us assume that $r<s.$ First suppose $d_r = -2.$ Then $d_s=-1$ or $j-2$. If $d_s = -1$, then $r-2=s-1$, which contradicts our assumption that $r<s.$ If $d_s=j-2$ , then $r-2=s+(j-2)$ and so $r=s+j$, contradicting $r<s.$

    Now assume $d_r=-1.$ Then $d_s\in\{-2,j-2\}.$ Note we cannot have $d_s = j-2$ since then we would have $r=s+j-1$ which violates $r<s$ since $j>1.$ Suppose instead that $d_s=-2.$ In this case, we have $r=s-1,$ implying that $-1$ occurs immediately before $-2$ in $d.$ 
    However, this does not happen since $d_r=-1$ implies that $r$ is in the partial sum of the composition opposite \texttt{last\_j} while $d_{r+1}=-2$ implies $r$ is not in $PS(\lambda)\cup PS(\mu),$ giving us another contradiction.  

    Finally, suppose $d_r = j-2$, in which case $d_s\in\{-2,-1\}.$ If $d_s = -2,$ then this means that $j-2+r = s-2$, and so $s = r+j$. Furthermore, since $d_s=-2,$ $s-1\not\in PS(\lambda)\cup PS(\mu).$ However, if $d_r = j-2$ for $r>1,$ we must have $r+j-1$ in the partial sum of \texttt{last\_j} corresponding to the $j$ added at that $r$-th step, giving us a contradiction since $r+j=s$.
    If $d_s = -1,$ then we have $r+j-1=s$, and have that $r+j-1\in PS(\lambda)\cap PS(\mu)$, contradicting that $\frac{\lambda}{\mu}$ is indecomposable. 
\end{proof}

 \begin{theorem}
     The maps $f$ and $g$ defined above are inverses of each other. That is, for any $\pi\in\Pi_{n-1}^{\{-2,-1,j-2\}}$, we have $g(f(\pi)) = \pi$ and for any $\frac{\lambda}{\mu}\in \bar{\mathcal{I}}_n^{\{1,j\}}$, we have $f(g(\frac{\lambda}{\mu})) = \frac{\lambda}{\mu}$.
 \end{theorem}

 \begin{proof}
    It is enough to show that the $-2$'s must appear in the correct places, since all other steps clearly ``undo'' the steps of the other map. 

    Suppose $\pi \in \Pi_{n-1}^{\{-2,-1,j-2\}}$ with $d_i(\pi)=-2$ for $i\in I.$ When computing $f(\pi)$, we can keep track of $PS(\lambda)$ and $PS(\mu).$ Indeed, $d_1=j-2$ places $j\in PS(\lambda)$ and $1 \in PS(\mu)$ and when see $d_r=j-2$ for $r>1$ we include $r+j-1$ in $PS(\texttt{last\_j})$. If $d_r = -1,$ we have $r-1$ in $PS(\texttt{last\_j})$ and if $d_r=-2,$ you do not add anything to the partial sum. Since $\pi$ is a permutation and so $\pi_r\neq\pi_{r-j}$,  this means that $r-1$ is not in $PS(\lambda)\cup PS(\mu).$ Thus when computing $g(f(\pi))$, the original displacement is recovered. It can similarly be argued that $f(g(\frac{\lambda}{\mu})) = \frac{\lambda}{\mu}.$
 \end{proof}


Since we have a bijection between indecomposable meanders with parts in $\{1,j\}$ and permutations with displacements in $\{-2,-1,j-2\}$, a natural question arises: Is there a statistic on permutations that corresponds (via this bijection) to the index of the meander? We show in the next theorem that when the meander is acyclic, the index of the meander can be determined from a simple statistic on the corresponding permutation.

\begin{theorem}\label{thm: -1s}
    Let $\frac{\lambda}{\mu}\in \bar{\mathcal{I}}_n^{\{1,j\}}$ 
    with $j\in\{2,3,4,5,7\}$ (so that $\{1,j\}$ is acyclic). Then if $\pi = g(\frac{\lambda}{\mu})$, where $g$ is defined in Definition \ref{def:g},   \[\ind\bigg(\frac{\lambda}{\mu}\bigg)=\frac{N_{o}(\pi)}{2}+1\] where $N_{o}(\pi)$ is the number of odd integers in $d(\pi)$, the displacement list of  $\pi$. 
\end{theorem}

\begin{proof}
    For an acyclic meander, the index is exactly the number of paths or isolated nodes. Let us first see that the number of paths and isolated nodes is equal to half of the total number of odd parts in both $\lambda$ and $\mu$. Indeed, an isolated node in $\frac{\lambda}{\mu}$ must either correspond to a 1 in both $\lambda$ and $\mu$ (that line up in the meander) or to an odd element in $\lambda$ or $\mu$ whose center node lines up with another 1 or center node in the other composition. A path, on the other hand, must have a starting node and ending node. Similarly, a starting node (or ending node) can only be either the center node of an odd part of $\lambda$ or $\mu$, or a 1 in $\lambda$ or $\mu.$

    Now, recall that the map $g$ builds the displacement list of a permutation $\pi$ by including a $-1$ in $d(\pi)$ for each of the 1's in the pair of compositions except for the first 1 and the last 1. Thus the displacement list will have exactly two fewer $-1$'s than the compositions have 1's. Additionally $g$ includes $j-2$ in the displacement list for each $j$ in $\lambda$ or $\mu$. Note that $j$ and $j-2$ have the same parity. Thus, the number of odd elements that appear in $\lambda$ and $\mu$  must be equal to $N_o+2$, and so the index of $\frac{\lambda}{\mu}$ must equal $\frac{N_o}{2}+1.$
    \end{proof}

We demonstrate Theorem \ref{thm: -1s} for both even and odd $j$ in the following two examples. 

\begin{example}
    Consider $\pi=3 1 5 2 4 8 6 7 (11)9 (10) \in \Pi_{11}^{\{-2,-1,2\}}$ with \[d(\pi)=
    (2, -1, 2, -2, -1, 2, -1, -1, 2, -1, -1).\] By Theorem \ref{thm: -1s} the associated indecomposable meander $\frac{\lambda}{\mu}\in\bar{\mathcal{I}}_{12}^{\{1,4\}}$ should have index $\frac{6}{2}+1=4$. We confirm this by applying $f$ from Definition~\ref{def:f} to obtain the meander $\frac{\lambda}{\mu}=\frac{4|1|4|1|1|1}{1|1|4|1|1|4}$. The four paths that give us index 4 are highlighted below. 
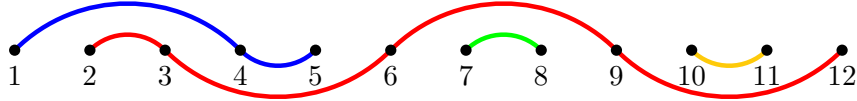
\begin{center}
\begin{tikzpicture}[
    dot/.style={circle, fill, inner sep=1.5pt},
    arc/.style={line width=0.6mm}
]

    \draw[arc, blue] (1,0) to[bend left=45] (4,0);
    \draw[arc, red] (2,0) to[bend left=45] (3,0);
    \draw[arc, red] (6,0) to[bend left=45] (9,0);
    \draw[arc, green] (7,0) to[bend left=45] (8,0);
    \draw[arc, red] (3,0) to[bend right=45] (6,0);
    \draw[arc, blue] (4,0) to[bend right=45] (5,0);
    \draw[arc, red] (9,0) to[bend right=45] (12,0);
    \draw[arc, yellow!60!orange] (10,0) to[bend right=45] (11,0);

 \foreach \x in {1,...,12} {
        \node[dot, label=below:{$\x$}] at (\x,0) {};
    }
\end{tikzpicture}
\end{center}

\end{example} 

\begin{example} Consider $\pi=416235(10)789\in \Pi_{10}^{\{-2,-1,3\}}$ with \[d(\pi)=
(3,-1,3,-2,-2,-1,3,-1,-1,-1).\]  By Theorem \ref{thm: -1s} the associated indecomposable meander $\frac{\lambda}{\mu}\in\bar{\mathcal{I}}_{11}^{\{1,5\}}$ should have index $\frac{8}{2}+1=5$. We confirm this by applying $f$ from Definition~\ref{def:f} to obtain the meander $\frac{\lambda}{\mu}=\frac{5|1|5}{1|1|5|1|1|1|1}$. We see there are 5 components: Four paths and one isolated vertex, which combine to give us index 5. 

    \begin{center}
\begin{tikzpicture}[
    dot/.style={circle, fill, inner sep=1.5pt},
    arc/.style={line width=0.6mm}
])
    
    \draw[arc, blue] (1,0) to[bend left=45] (5,0);
    \draw[arc, red] (2,0) to[bend left=45] (4,0);
    \draw[arc, green] (7,0) to[bend left=45] (11,0);
    \draw[arc, yellow] (8,0) to[bend left=45] (10,0);
    \draw[arc, green] (3,0) to[bend right=45] (7,0);
    \draw[arc, red] (4,0) to[bend right=45] (6,0);
    
    \foreach \x in {1,...,11} {
            \node[dot, label=below:{$\x$}] at (\x,0) {};
        }
\end{tikzpicture}
\end{center}
\end{example}

\section{Future work}\label{sec: conclusion}

\subsection{Enumeration for non-acyclic sets}

An obvious open question is to find a generating function $A_J(x,y) = \sum a_{n,k}^Jx^ny^k$ for sets $J$ not covered in this paper. In particular, it seems difficult to consider cases when $J$ is not acyclic. We do address one such non-acyclic case in this paper: $\{1,2,n\}$. Reasonable sets to consider next might be $\{1,3,n\}$, $\{1,2,n-1\}$, or similar sets. It may also be reasonable to attempt enumerating $\{1,6\}$ by index.

\subsection{Permutations with other displacement sets}\label{subsec: morerelate}

Section~\ref{sec: bijection} tells us there is a bijection between permutations $\pi\in\S_{n-1}$ with displacements in $D = \{-2,-1,j-2\}$ and indecomposable meanders $\frac{\lambda}{\mu}$ with $\lambda_1>\mu_1$ so that the elements of $\lambda$ and $\mu$ are taken from the set $J = \{1,j\}.$ A natural question to ask is whether we can change the sets $D$ and $J$ in a predictable way to obtain similar bijections. We do suspect this to be the case. For example, it is possible to construct a bijection between permutations $\pi \in \S_{n-2}$ with displacements restricted to $\{-2,0,1\}$ and $\pi_1\neq 1$ and indecomposable meanders with part sizes restricted to $\{2,3\}$. It is similarly possible to construct a bijection between permutations $\pi \in \S_{n-7}$ with displacements restricted to $\{-2,1,2\}$ and indecomposable meanders with part sizes restricted to $\{3,4\}$. (Note this also follows from Theorems~\ref{thm: 14Indec} and \ref{thm: 34} together with the reverse-complement of the corresponding permutation.) It would be interesting if there were a general answer to the question of how indecomposable meanders with restricted part sizes relate to permutations with restricted displacements.

\subsection{Alternating sums by index}

In \cite{SY20}, the authors prove a conjecture from \cite{CHMW15} regarding the alternating sum $\sum (-1)^k c_{n,k}$ where $c_{n,k}$ is the number of meanders of index $k$ where $\lambda$ and $\mu$ are taken to be partitions. 
Below we consider a similar alternating sum for meanders of compositions with parts in $\{1,2\}.$

 \begin{theorem}
    \[
    \sum_{k=1}^n (-1)^k a_{n,k}^{\{1,2\}} = 
    \begin{cases}
    0 & n \equiv 2,5 \pmod 6\\
    1 & n \equiv 0,4 \pmod 6\\
    -1 & n \equiv 1,3 \pmod 6
    \end{cases}
    \]
\end{theorem}

\begin{proof}
    We can establish an almost-bijection $\varphi$ from meanders in $\M_n^{\{1,2\}}$ with even index to meanders in $\M_n^{\{1,2\}}$ with odd index that
     fails in exactly one case when $n \not\equiv 2 \pmod 3$.

     For $\frac{\lambda}{\mu}$ with even index, define $\frac{\lambda'}{\mu'} = \varphi(\frac{\lambda}{\mu})$ recursively as follows. 
    If $\lambda = (1,1,\lambda_3,\ldots,\lambda_r)$, then take $\lambda'=(2,\lambda_3,\ldots, \lambda_r)$ and $\mu'=\mu$. Similarly, if $\lambda = (2,\lambda_2,\ldots,\lambda_r)$, then take $\lambda'=(1,1,\lambda_2,\ldots, \lambda_r)$ and $\mu'=\mu$.  Note that in each case, you change the parity of the index. If $\mu_1=2,$ then in the first case you increase the index by 1 since you replace a path with a cycle, and in the second case you decrease the index by 1 since you replace a cycle with a path. If $\mu_1=1,$ then in the first case you decrease the index by 1 since you remove an isolated point and in the second case, you increase the index by 1 since you add an isolated point.

    If $\lambda = (1,2,\ldots)$ and $\mu$ is $(1,1, \mu_3,\ldots, \mu_s)$ or $(2,\mu_2,\ldots, \mu_s)$, then do the same as above to $\mu$ instead of $\lambda.$ The parity of the index is similarly changed. 
    If both $\lambda$ and $\mu$ start with $1,2$, then apply the bijection to $\bar\lambda,\bar\mu$, each obtained by deleting the $1,2$ at the beginning. Let $\varphi(\frac{\lambda}{\mu}) = \frac{1|2| \bar\lambda'}{1|2|\bar\mu'}.$ 
    
    Note this bijection only fails if both $\lambda$ and $\mu$ are of the form $(1,2,1,2,1,2,1,\ldots,2)$ (when $n\equiv 0 \bmod 3$) or $(1,2,1,2,1,2,\ldots1)$ (when $n\equiv 1 \bmod 3$). In each of these special cases, you can compute the index and obtain the result above. 
\end{proof}

A similar question could easily be asked for meanders whose compositions appear in some fixed set $J.$ Numerical evidence suggests the answers may be reasonable to find.

\subsection*{Declarations}

The views expressed in this article are those of the authors and do not reflect the official policy or position of the U.S. Naval Academy, Department of the Navy, the Department of Defense, or the U.S. Government.

The authors declare that no funds, grants, or other support were received during the preparation of this manuscript. The authors have no relevant financial or non-financial interests to disclose. Data sharing is not applicable to this article as no datasets were generated or analyzed during the current study.

\bibliographystyle{amsplain}

\end{document}